\documentclass[11pt]{amsart}

\usepackage[english]{babel}
\usepackage[utf8x]{inputenc}
\usepackage[T1]{fontenc}
\usepackage{physics}
\usepackage{fullpage}

\usepackage{amsmath}
\usepackage{amssymb}
\usepackage{amsthm}
\usepackage{graphicx}
\usepackage{subcaption}
\usepackage[numbers]{natbib}
\usepackage[colorlinks=true, allcolors=blue]{hyperref}

\newtheorem{theorem}{Theorem}[section]
\newtheorem{corollary}[theorem]{Corollary}

\newtheorem{proposition}[theorem]{Proposition}
\newtheorem{lemma}[theorem]{Lemma}
\newtheorem{remark}[theorem]{Remark}
\newtheorem*{theorem*}{Theorem}
\newtheorem{definition}[theorem]{Definition}
\newtheorem{example}[theorem]{Example}
\newtheorem{construction}[theorem]{Construction}

\DeclareMathOperator{\Ad}{Ad}

\newcommand{\bB}{{\mathbb{B}}}
\newcommand{\bC}{{\mathbb{C}}}

  \newcommand{\A}{{\mathcal{A}}}
  \newcommand{\B}{{\mathcal{B}}}
  
  \newcommand{\D}{{\mathcal{D}}}
  
  \newcommand{\F}{{\mathcal{F}}}
  
\renewcommand{\H}{{\mathcal{H}}}
  \newcommand{\I}{{\mathcal{I}}}
  \newcommand{\J}{{\mathcal{J}}}
  \newcommand{\K}{{\mathcal{K}}}  
\renewcommand{\L}{{\mathcal{L}}}
  \newcommand{\M}{{\mathcal{M}}}
  
\renewcommand{\O}{{\mathcal{O}}}

  \newcommand{\R}{{\mathcal{R}}}
\renewcommand{\S}{{\mathcal{S}}}
  \newcommand{\T}{{\mathcal{T}}}
  
  \newcommand{\V}{{\mathcal{V}}}
  \newcommand{\W}{{\mathcal{W}}}
  \newcommand{\X}{{\mathcal{X}}}
  \newcommand{\Y}{{\mathcal{Y}}}
  \newcommand{\Z}{{\mathcal{Z}}}

\title{Products and factorization in operator systems}

\author{Adam Dor-On}
\author{Travis Russell}

\address{Department of Mathematics, University of Haifa, Mount Carmel, Haifa 3103301, Israel}
\email{adoron.math@gmail.com}

\address{Department of Mathematics, Texas Christian University, Fort Worth, TX 76129}
\email{travis.b.russell@tcu.edu }



\thanks{A. Dor-On was partially
supported by an NSF-BSF grant no. 2350543 / 2023695 (respectively). Both authors were partially supported by a BSF Start-up grant no. 2024161}

\begin{document}

\maketitle

\begin{abstract}
    We study unital operator spaces endowed with a partially defined product. We give a matrix-norm characterization of such products that allows for a representation theorem where the partial product is realized as composition of operators on Hilbert space. We study product-respecting C*-covers, including a universal \emph{product} C*-cover, and \emph{product} quotients. We show that for the Haagerup tensor product of unital operator spaces remains injective, while projectivity holds relative to \emph{product} quotients. Moreover, we identify the commuting tensor product as a complete \emph{product} quotient of the Haagerup tensor product. Our framework yields new factorization norm formulas for a variety of product structures, as well as an intrinsic trace-extension criterion that resolves a question posed by Sinclair. Our work unifies and extends tensor products for operator systems, with applications in quantum information theory.
\end{abstract}

\section{Introduction}

The theory of operator systems was first studied for its own sake by Arveson \cite{Arveson69}, who viewed them as noncommutative analogues of function systems. It has since become a central framework for structural and geometric aspects of operator algebras \cite{paulsen2002completely, KAVRUK2011}, with deep connections to complexity and quantum information theory \cite{ji2020mip, paulsen2016estimating}. A recurring theme in this theory is the study of tensor products of operator systems \cite{KAVRUK2011}, which encode subtle interactions between positivity and composition of operators. Classical examples such as the minimal, commuting, and Haagerup tensor products of operator spaces are tied to major open problems, including Connes’ embedding problem \cite{junge2011connes, ozawa2013connes}, the Smith–Ward problem \cite{kavruk2014nuclearity}, and Kadison’s similarity problem via Pisier’s similarity degree \cite{Pisier01Similarity}. Yet despite the far-reaching applications these constructions provide, fundamental challenges still remain in obtaining explicit descriptions, particularly in the finite-dimensional setting.

While there is a well-established tensor product theory for operator systems (including the minimal and commuting tensor products), a broader perspective emerges when one attempts to define a Haagerup tensor product at the operator \emph{system} level. This leads us to consider \emph{abstract products} of unital operator systems that go beyond the standard tensor product framework. The resulting viewpoint captures a richer class of examples and supports a systematic treatment of multiplicative structures in operator systems and unital operator spaces.

A central motivating example for this work is the Haagerup tensor product. Given two unital operator spaces $\S$ and $\T$, their Haagerup tensor product $\S \otimes_h \T$ is the “freest” unital operator space product they generate. This imposes the fewest restrictions on how $\S$ and $\T$ multiply while retaining their operator space structure. However, the partial multiplication on $\S \otimes_h \T$ does not behave well with the universal C*-algebra functor. Indeed, in Proposition~\ref{p:not-prod-cov} we show that the canonical quotient
\[
C^*_u(\S \otimes_h \T)\;\longrightarrow\; C^*_u(\S)\,*_1 \,C^*_u(\T)
\]
need not be injective. This already indicates that the additional partial multiplication on $S\otimes_h T$ must be taken into account in order to obtain natural functorial properties for a unital operator space or an operator system analogue of the Haagerup tensor product.

For operator \emph{spaces}, the Haagerup tensor product is both injective and projective. In the \emph{unital} operator space setting, it is relatively straightforward to show that injectivity persists in the expected sense. That is, if $\phi: \S\to \S'$ and $\psi: \T\to \T'$ are unital complete isometries, then
\[
\phi\otimes\psi:\; \S \otimes_h \T \;\longrightarrow\; \S' \otimes_h \T'
\]
is a unital complete isometry. By contrast, in Theorem \ref{t:product-proj-haagerup} we prove projectivity \emph{relative to maps that respect the partial product} on the Haagerup tensor product. That is, if $\J \subseteq \S$ and $\K \subseteq \T$ are kernels of unital completely positive maps, the induced product map
\[
[(\S \otimes_h \T)/(\J \otimes \T+\S \otimes \K)]_m\;\longrightarrow\;(\S / \J)\otimes_h(\T / \K)
\]
is a unital complete isometry when the quotient taken in the left-hand side is a \emph{product} operator system quotient. This “\emph{product projectivity}” is the appropriate projective property once a partial multiplication is built into the structure of the tensor product.

To address these issues in a unified way, in Section \ref{s:product-structure} we introduce a framework for capturing partial multiplication on unital operator spaces. In this setup, an abstract product on a unital operator space is specified by a completely contractive, partially defined multiplication on a set of valid pairs for the underlying space. Specifically, for an operator system $\V$, we let $\D \subseteq \V \times \V$ denote the set of pairs $(x,y)$ for which the abstract product, denoted $m(x,y)$, is an element of $\V$. Corollary \ref{cor: abstract product characterization} is our main realization result, which shows that every abstract product of unital operator spaces admits a unital completely isometric realization in which the partial multiplication is realized by operator composition; i.e., we characterize when there exists a Hilbert space $\H$ and a complete isometry $\pi: \V \to \bB(\H)$ such that $\pi(m(x,y)) = \pi(x) \pi(y)$ for all pairs $(x,y) \in \D$.

Our abstract characterization of products of unital operator spaces naturally leads to us to define C*-covers and quotients that respect the intrinsic partial multiplication. While in Theorem \ref{thm: pp rep in C*-env} we show that the C*-envelope captures a minimal concrete representation of a product, the universal C*-algebra of an operator system may fail to be a product-respecting cover. We therefore introduce the universal product C*-cover $C^*_m(\V)$, which encodes the partial product structure of a given product operator space $\V$ as a product taken in the C*-cover. Quotients enter the picture via \emph{product quotients} in the sense that if a completely positive product map has kernel $\K$, the corresponding quotient $\V / \K$ inherits a well-defined partial multiplication on the space making it into a product. In particular, this viewpoint allows us to show in Theorem \ref{t:commutig-is-haagerup-prod-quotient} that $\S \otimes_c \T$ is a complete \emph{product} quotient of $\S \otimes_h \T$, providing a structural bridge between these two fundamental tensor products.

In these respects, our approach is a natural continuation of the program initiated by Kavruk, Paulsen, Todorov and Tomforde \cite{KAVRUK2011}, whose systematic treatment of tensor products of operator systems had far-reaching impact in quantum information and operator algebra theory \cite{ji2020mip}. Whereas \cite{KAVRUK2011} organizes tensor products via the use of matrix cones and mapping properties for ucp maps, our product-centric framework encodes multiplicative data intrinsically through matrix norms and partially defined completely contractive products. In turn, this yields intrinsic factorization norm descriptions. For instance in Corollary \ref{cor: commuting tensor norm} we uncover new factorization norms for $\S\otimes_c \T$, in Corollary \ref{cor: Quotient factorization norm} we find that product quotients recover ordinary operator system quotients when the partial multiplication is trivial, and in Theorems \ref{t:factorization-norm-group} and Corollary \ref{c:correlation-system-factorization} we discover new factorization norm formulas for operator systems arising from groups and projection-valued measures. Together these tools provide practical ways to compute and compare norms across commuting tensor products, operator system quotients, and in natural settings related to correlation sets in quantum information theory. For example, by considering operator systems spanned by projections, we find matrix-norm analogues for an operator system construction of Araiza and the second author \cite{AraizaRussellAbs23} which produces a universal operator system whose state space is in one-to-one correspondence with the set of all quantum commuting correlations.

Finally, we resolve a question posed by Thomas Sinclair in the November 2019 conference ``Quantitative Linear algebra meets Quantum Information Theory'' about abstractly characterizing traces on operator systems. We provide in Corollary \ref{c:tracial-extension} intrinsic conditions for the existence of a \emph{tracial} extension of a state on a product space $\V$ to $C^*_m(\V)$. More precisely, let $\V$ be an operator system with a completely contractive partial product $m:\D \to \V$. Working in the free algebra $\mathcal{F}_m(\V)$ generated by $\V$ and the permissible $m$-factorizations, we define a natural factorization seminorm $\| \cdot \|_{\tau}$. We show that a state $\varphi:\V \to\bC$ extends to a tracial state on the universal product C*-cover $C_m^*(\V)$ if and only if $\varphi$ is contractive for the restriction of the seminorm $\| \cdot \|_{\tau}$ to $\V$. Thus, without having to restrict a trace from a tracial C*-cover for $\V$, traces on $\V$ can be considered \emph{abstractly} as those states on $\V$ that admit tracial extensions to $C^*_m(\V)$, which in turn can be determined \emph{intrinsically} from a factorization norm related to the partial multiplication $m$ prescribed on $\V$. As an application, we show that tracial states on the universal operator system for quantum commuting correlations are in one-to-one correspondence with corners of synchronous quantum commuting correlations. Thus we have an operator space characterization of the set of synchronous quantum commuting correlations, whereas previously known characterizations from the literature require tracial states on C*-algebras.

The remainder of the paper is organized as follows. In Section \ref{s:prelim} we provide some preliminary results from the literature for the reader's convenience. In Section \ref{s:haagerup-tensor} we study the Haagerup tensor product in the context of unital operator spaces and operator systems, and explain its behaviour with respect to the universal C*-algebra and the C*-envelope. In Section \ref{s:product-structure} we define products of unital operator spaces and provide a new abstract characterization for them. In Section \ref{s:universal-covers-quotients-prod} we introduce product C*-covers and product quotients in our newly defined category of unital operator spaces with multiplicative structure. Finally, in Section \ref{s:factorization-application} we introduce factorizations norms and provide applications in a variety of settings. This includes operator system quotients, commuting tensor products, group operator systems, quantum correlation sets and their synchronous corners, and finally a characterization of states on operator systems that extend to traces on the universal product C*-cover. We provide many specific examples and constructions throughout the paper to motivate our definitions and results.

\subsection*{Acknowledgments} The authors are grateful for the hospitality of the American Institute of Mathematics, where work on this paper took place as part of the SQuaRE meetings on ``Approximation Properties for Operator Systems and Matrix Convex Sets''. The authors also thank Thomas Sinclair for discussions on the content of this manuscript.

\section{Preliminaries} \label{s:prelim}

For the convenience of the reader, we review some fundamental concepts and historical remarks concerning operator spaces and operator systems that appear throughout the paper. For a more thorough introduction to these topics, we recommend the textbooks \cite{BlecherOperatorAlgebras} and \cite{paulsen2002completely}.

An \textit{operator space} is defined concretely as a subspace of $\bB(\H)$ for some Hilbert space $\H$, whereas an \textit{operator system} is defined concretely as a unital self-adjoint subspace of $\bB(\H)$. Operator spaces were abstractly characterized as vector spaces equipped with $\mathrm{L}^{\infty}$-matrix norms by Ruan in \cite{RUAN1988}. Specifically, an $\mathrm{L}^{\infty}$-\textit{matrix seminorm} on a vector space $\V$ is a sequence of seminorms $\|\cdot\|_{n,m}: M_{n,m}(\V) \to [0,\infty)$ which satisfy 
\begin{itemize}
    \item $\|\alpha x \beta\|_{k,l} \leq \|\alpha\| \|x\|_{n,m} \|\beta\|$ for all $x \in M_{n,m}(\V)$, $\alpha \in M_{k,n}$, and $\beta \in M_{m,l}$, and
    \item $\|x \oplus y\|_{n+m,k+l} = \max(\|x\|_{n,k},\|y\|_{m,l})$ for every $x \in M_{n,k}(\V)$ and $y \in M_{m,l}(\V)$.
\end{itemize}
If each $\|\cdot\|_{n,m}$ is a norm, then $\{\|\cdot\|_{n,m}\}$ is called an $\mathrm{L}^{\infty}$-\textit{matrix norm}. An \textit{abstract operator space} is a complex vector space $\V$ equipped with an $\mathrm{L}^{\infty}$-matrix norm $\{\|\cdot\|_{n,m}\}$.

When $\{\|\cdot\|_{n,m}\}$ is an $\mathrm{L}^{\infty}$-matrix seminorm, it is automatically the case that each $\|\cdot\|_n:=\|\cdot\|_{n,n}$ is a seminorm. Indeed, if $\lambda \in \bC \setminus \{0\}$ and $A \in M_n(\X)$, then
\[ \|\lambda A \|_n \leq \| \lambda I_n\| \|A\|_n = |\lambda| \|A\|_n \quad \text{and} \quad  \| A\|_n = \|\lambda^{-1} \lambda A\|_n \leq \|\lambda^{-1}I_n \| \|\lambda A\|_n = |\lambda|^{-1} \|\lambda A\|_n \]
so that $\|\lambda A\|_n = |\lambda| \|A\|_n$. Also if $A, B \in M_n(\V)$ with $\|A\|_n \neq 0$ and $\|B\|_n \neq 0$, then
\[ \|A+B\|_n = \| \begin{pmatrix} (\|A\|_n)^{1/2} \\ (\|B\|_n)^{1/2} \end{pmatrix}^T \begin{pmatrix} \frac{A}{\|A\|_n} & 0 \\ 0 &  \frac{B}{\|B\|_n} \end{pmatrix} \begin{pmatrix} (\|A\|_n)^{1/2} \\ (\|B\|_n)^{1/2} \end{pmatrix} \|_{n} \leq \| \begin{pmatrix} (\|A\|_n)^{1/2} \\ (\|B\|_n)^{1/2} \end{pmatrix}\|^2 = \|A\|_n + \|B\|_n.  \]
It obvious that $\|\lambda A\|_n = |\lambda| \|A\|_n$ and $\|A + B \|_n \leq \|A\|_n + \|B\|_n$ when $\lambda = 0$ or $\|A\|_n = 0$ or $\|B\|_n = 0$. Hence, $\|\cdot\|_n$ is a seminorm.  Furthermore, it is well-known that to prove $\{\|\cdot\|_{n,m}\}$ is an $\mathrm{L}^\infty$-matrix seminorm, it suffices to show that $\{\|\cdot\|_n\}$ is an $\mathrm{L}^\infty$-matrix seminorm, i.e. it suffices to consider square matrices. Therefore we often only consider norms on square matrices over operator spaces.

Operator systems had been characterized earlier by Choi and Effros as matrix-ordered $*$-vector spaces equipped with an Archimedean matrix order unit in \cite{CHOIEffros1977}. While it is clear that every concrete operator system is a concrete operator space, it is also true that every abstract operator system is an abstract operator space with $\mathrm{L}^{\infty}$-matrix norms uniquely defined in terms of the matrix ordering and order unit. We will not make use of matrix orderings in this paper, so we refer the reader to \cite[Chapter 13]{paulsen2002completely} for more details.

Given operator spaces $\V$ and $\W$ and a linear map $\varphi: \V \to \W$, we define $\varphi^{(n,m)}: M_{n,m}(\V) \to M_{n,m}(\W)$ by $\varphi^{(n,m)}((a_{i,j})) = (\varphi(a_{i,j}))$, i.e. apply $\varphi$ to the entries of the matrix. If the set of maps $\varphi^{(n,m)}$ are uniformly bounded in $n,m$, we say that $\varphi$ is \textit{completely bounded} and define $\|\varphi\|_{cb}$ to be the smallest such uniform bound. We say $\varphi$ is \textit{completely contractive} if $\|\varphi\|_{cb} \leq 1$. In the case when $\V$ and $\W$ are operator systems, we say that $\varphi$ is \textit{completely positive} if $\varphi^{(n)} := \varphi^{(n,n)}$ is positive for every $n$. Owing in part to the abstract characterizations of Ruan, Choi and Effros, operator spaces are frequently regarded as the objects of a category whose morphisms are completely bounded or completely contractive maps, whereas operator systems are regarded as the objects of a category whose morphisms are completely positive or unital completely positive (abbreviated ucp) maps. A linear map $\varphi: \V \to \W$ is a \textit{complete isometry} if each map $\varphi^{(n,m)}$ is isometric. When $\V$ and $\W$ are operator systems, then $\varphi$ is called a \textit{complete order embedding} if both $\varphi$ and $\varphi^{-1}$ (restricted to the range of $\varphi$) are completely positive. It is well-known that completely positive maps on operator systems are automatically completely bounded, ucp maps are automatically completely contractive, and that a unital linear map is a complete order embedding if and only if it is a completely isometry (see \cite[Chapter 3]{paulsen2002completely}).

With the terminology developed above, we can now relate concrete and abstract operator spaces and operator systems more succinctly. Recall that if $\H$ is a Hilbert space, then $\bB(\H)$ is an operator system with matrix norms defined by identifying $M_{n,m}(\bB(\H))$ with $\bB(\H^m, \H^n)$ isometrically.

\begin{theorem}[Ruan, Choi-Effros] \label{t:RCE}
    Let $\V$ be an abstract operator space (respectively, an operator system) with matrix norm $\{\|\cdot\|_{n,m}\}$. Then there exists a Hilbert space $\H$ and a complete isometry (respectively, a unital complete isometry) $\varphi: \V \to \bB(\H)$.
\end{theorem}

In what follows, if the domain $M_{n,m}(\V)$ is clear from context, we will sometimes denote $\|x\|_{n,m}$ simply by $\|x\|$ or $\varphi^{(n,m)}(x)$ simply by $\varphi(x)$ for $x \in M_{n,m}(\V)$.

Abstract characterizations, like the ones above for operator spaces and operator systems, provide a more flexible setting than their concrete counterparts for defining operator structures on various algebraic constructions. For instance, suppose $\V$ is an operator space and $\J \subseteq \V$ is a closed subspace. Then we can define an operator space structure on the vector space $\V / \J$ by setting \[ \|x + \J\|_{n,m}^{osp} := \inf \{ \|x + j \|_{n,m} : j \in M_{n,m}(\J) \}. \] We denote the resulting operator space $(\V / \J, \{\|\cdot\|_{n,m}^{osp}\})$ by $(\V / \J)_{osp}$. It is routine to check that $(\V / \J)_{osp}$ is an abstract operator space. On the other hand, if $\V \subseteq \bB(\H)$ is a given concrete operator space with closed subspace $\J$, it is not obvious how one realizes $\V / \J$ as a concrete operator space without appealing to Theorem \ref{t:RCE}.

Similarly, if $\V$ is an operator system with unit $e$, and if $\J \subseteq \V$ is the kernel of a ucp map, then it is possible to define an operator system structure on the quotient vector space $\V / \J$. The unit of the quotient operator system is $e + \J$. The matrix ordering on this quotient is defined in \cite{Kavruk}, but we will not need to recall that definition here. Instead, we will work with an equivalent definition stated in terms of the matrix norms. Specifically, for $x+\J \in M_{n,m}(\V / \J)$ we define,
\[ \|x + \J\|_{n,m}^{osy} := \sup \{ \| \varphi(x)\| : \varphi: \V \to \bB(\H) \text{ is ucp and } \J \subseteq \ker(\varphi) \}. \]
We let $(\V / \J)_{osy}$ denote the resulting operator system. Although we always have the inequality $\|x\|^{osy} \leq \|x\|^{osp}$, it was shown in \cite{Kavruk} that the matrix norms on $(\V / \J)_{osp}$ and $(\V / \J)_{osy}$ are generally not equal.

Many results stated throughout this paper hold for operator spaces which are unital but not necessarily self-adjoint. We will discuss abstract characterizations for such spaces below in Subsection \ref{subsec: unital op sp}. For the time being, we say that an abstract operator space $\V$ is \textit{unital} with unit $e\in \V$ if $e$ is a unit vector and there is a complete isometry $\varphi: \V \to \bB(\H)$ for some Hilbert space $\H$ such that $\varphi(e) = I_{\H}$. Given a concrete untial operator space $\V$ and a unital completely contractive map (abbreviated ucc) $\varphi: \V \to \bB(\H)$, there exists a unique ucp extension $\varphi': \V + \V^* \to \bB(\H)$ defined by $\varphi'(x + y^*) := \varphi(x) + \varphi(y)^*$ (see \cite[Proposition 3.4]{paulsen2002completely}). It follows that any abstract unital operator space $\V$ (in the sense defined above) generates a unique operator system $\V + \V^*$ (c.f. \cite{RUSSELL2017} for a more technical explanation of this fact). For these reasons, structures defined for a unital operator space $\V$ can often be related in a unique way to a corresponding structure on the operator system $\V+\V^*$. For example, suppose that $\J \subseteq \V$ is the kernel of a ucc map. Then we define 
\[ \|x + \J\|_{n,m}^{usp} := \sup \{ \| \varphi(x)\| : \varphi: \V \to \bB(\H) \text{ is ucc and } \J \subseteq \ker(\varphi) \}. \]
It is routine to check that $\V / \J$ equipped with this matrix norm is a unital operator space with unit $e+\J$. We denote this operator space as $(\V / \J)_{usp}$. Furthermore, since any ucc map $\varphi: \V \to \bB(\H)$ with $\J \subseteq \ker(\varphi)$ extends uniquely to a ucp map $\varphi': \V+\V^* \to \bB(\H)$ with $\J+\J^* \subseteq \ker(\varphi')$, we observe that the matrix norm on $(\V / \J)_{usp}$ is just the restriction of the matrix norm on $((\V+\V^*)/(\J+\J^*))_{osy}$ to the subspace $\V + (\J+\J^*)$, which in turn is linearly isomorphic to the vector space $\V / \J$.

We make use of a variety of C*-covers throughout the paper. Given a unital operator space $\V$, a \textit{C*-cover} is a C*-algebra $\A$ together with a unital completely isometric map $j: \V \to \A$ such that $\A = C^*(j(\V))$. In other words, a C*-cover is a C*-algebra generated by the image of $\V$ under some unital complete isometry. In particular, $\A$ is always a unital C*-algebra. Among the set of all C*-covers, there are two principle C*-covers that are frequently used. The first of these is the \textit{universal C*-cover} denoted $C^*_u(\V)$. This C*-cover is the ``largest'' in the sense that it satisfies the following universal property: if $\varphi: \V \to \bB(\H)$ is a ucc map, then $\varphi$ extends uniquely to a $*$-homomorphism $\pi_{\varphi}: C^*_u(\V) \to \bB(\H)$. Since ucc maps on a unital operator space $\V$ extend uniquely to ucp maps on the operator system $\V + \V^*$, we see that $C^*_u(\V) = C^*_u(\V+\V^*)$. Kirchberg and Wasserman were the first to study the universal C*-cover for an operator system in \cite{kirchberg1998c} and showed that it is residually finite dimensional. The other principle C*-cover we will use is called the \textit{C*-envelope}, denoted $C^*_e(\V)$. This is the ``smallest'' C*-cover in the sense that it satisfies the following co-universal property: if $j: \V \to C^*(\V)$ is a unital completely isometric map and $i: \V \to C^*_e(\V)$ is the C*-cover map for the C*-envelope, then there exists a unique $*$-homomorphism $\pi: C^*(\V) \to C^*_e(\V)$ such that $\pi(j(x)) = i(x)$ for all $x \in \V$. The existence of the C*-envelope was conjectured by Arveson in his seminal work \cite{Arveson69} and established by Hamana in \cite{hamana1979injective}. 

Hamana's proof established the existence of another important C*-algebra associated to a unital operator space $\V$ which is called the \textit{injective envelope}, and is denoted by $I(\V)$. A unital operator space $\S$ is called \textit{injective} provided that whenever $\V \subseteq \W$ is an inclusion of unital operator spaces and $\varphi: \V \to \S$ is ucc, then there exists a ucc extension $\varphi': \W \to \S$. For instance, Arveson's Extension Theorem, established in \cite{Arveson69}, shows that the unital operator space $\bB(\H)$ is injective. Given a unital operator space $\V$, there exists a unique smallest injective operator system $I(\V)$ that contains $\V$ isometrically. This operator system turns out to be unitally completely isometrically isomorphic to a C*-algebra. Moreover, if $i: \V \to I(\V)$ is the unital completely isometric embedding of $\V$ into $I(\V)$, then $C^*(i(\V)) = C^*_e(\V)$ (see \cite[Chapter 15]{paulsen2002completely} for more details).

We conclude with a summary of tensor products that appear in several discussions throughout the paper. Given a pair of operator spaces $\V$ and $\W$, a \textit{tensor product structure} is an operator space structure on the vector space $\V \otimes \W$ (i.e. an $\mathrm{L}^{\infty}$-matrix norm on $\V \otimes \W$) such that $\|x \otimes y\| = \|x\| \|y\|$ whenever $x \in \V$ and $y \in \W$ and such that $\varphi \otimes \psi: \V \otimes \W \to M_{nk}$ is completely contractive whenever $\varphi: \V \to M_n$ and $\psi: \W \to M_k$ are completely contractive. When $\V$ and $\W$ are unital with units $e\in \V$ and $f\in \W$, a tensor product structure on $\V \otimes \W$ is one making $\V \otimes \W$ an abstract unital operator space with unit $e \otimes f$. If $\V$ and $\W$ are operator systems with units $e\in \V$ and $f\in \W$, then an \textit{operator system tensor product structure} on $\V \otimes \W$ is an operator system structure on $\V \otimes \W$ for which $e \otimes f$ is unit, together with a conjugate linear adjoint satisfying $(v \otimes w)^* = v^* \otimes w^*$ for all $v \in \V$ and $w \in \W$ (see \cite{KAVRUK2011} for the precise definition). 

As a first example, suppose that we are given complete isometries $\pi: \V \to \bB(\H)$ and $\rho: \W \to \bB(\K)$ where $\H$ and $\K$ are Hilbert spaces. This induces a linear map $\pi \otimes \rho: \V \otimes \W \to \bB(\H \otimes \K)$ in the obvious way. The matrix norm induced on $\V \otimes \W$ by the map $\pi \otimes \rho$ is called the \textit{minimal tensor product} on $\V \otimes \W$ and is denoted by $\V \otimes_{min} \W$. It is somewhat surprising that this matrix norm does not depend on the choice of complete isometries $\pi$ and $\rho$. In fact, if $\{\|\cdot\|_{n,m}\}$ is an arbitrary operator space tensor product structure on $\V \otimes \W$, then for $x \in M_{n,m}(\V \otimes \W)$ we have $\|x\|_{min} \leq \|x\|$. It also turns out that whenever $\V$ and $\W$ are operator systems, $\V \otimes_{min} \W$ is an operator system, i.e. the maps $\pi$ and $\rho$ can be chosen to be unital, so that $e \otimes f$ is mapped to $I_{\H} \otimes I_{\H}$ under $\pi \otimes \rho$.

For a second example, suppose that $\V$ and $\W$ are operator systems. Let $i_1: \V \to C^*_u(\V)$ and $i_2: \W \to C^*_u(\W)$ denote the inclusions into the universal C*-covers. Then the linear map $i_1 \otimes i_2: \V \otimes \W \to C^*_u(\V) \otimes_{max} C^*_u(\W)$ defines matrix norms on $\V \otimes \W$. By identifying $\V \otimes \W$ with its image $(i_1 \otimes i_2)(\V \otimes \W)$ inside $C^*_u(\V) \otimes_{max} C^*_u(\W)$, we see that this is an operator system tensor product structure on $\V \otimes \W$, which we denote as $\V \otimes_c \W$. It turns out that the induced matrix norm agrees with the following: Given $x \in M_{n,m}(\V \otimes \W)$, define
\[ \|x\|^c := \sup \{ \|\varphi \cdot \psi (x) \| \} \]
where $(\varphi \cdot \psi)( a \otimes b) := \varphi(a) \psi(b)$, and the supremum is over all Hilbert spaces $\H$ and pairs ucp maps $\varphi: \V \to \bB(\H)$ and $\psi: \W \to \bB(\H)$ with commuting ranges. This tensor product, first introduced in \cite{KAVRUK2011}, has played an important role in many pivotal questions of recent interest in operator algebras, perhaps most notably through its connection to the recent solution to Connes' Embedding Problem \cite{ji2020mip}. 

Finally, we discuss an operator space tensor product that will have special relevance for unital operator spaces. We first recall the definition of the Haagerup tensor product of operator spaces, as well as its fundamental properties. Let $\X$ and $\Y$ be operator spaces. Given $(x_{ij}) \in M_{nk}(\X)$ and $(y_{ij}) \in M_{km}(\Y)$, we define
\[ (x_{ij}) \odot (y_{ij}) := ( \sum_{l=1}^k x_{il} \otimes f_{lj}) \in M_{nm}(\X \otimes \Y). \]
Given $(z_{ij}) \in M_{nm}(\X \otimes \Y)$, we define
\[ \|(z_{ij})\|_h := \inf \{ \|(x_{ij})\| \|(y_{ij})\| : (z_{ij}) = (x_{ij}) \odot (y_{ij}) \} \]
and we let $\X \otimes_h \Y$ denote the tensor product $\X \otimes \Y$ equipped with the matrix norms $\|\cdot\|_h$.

The Haagerup tensor product is known to be a tensor product in the category of operator spaces. It has the following celebrated properties.

\begin{theorem}
Let $\X,\Y,\Z,\W$ be operator spaces.
\begin{enumerate}
    \item Suppose $\phi: \X \to \Z$ and $\psi: \Y \to \W$ are complete isometries. Then $\phi \otimes \psi: \X \otimes_h \Z \to \Y \otimes_h \W$ is a complete isometry (i.e. $\otimes_h$ is injective).
    \item Suppose $\phi: \X \to \Z$ and $\psi: \Y \to \W$ are complete quotient maps. Then $\phi \otimes \psi: \X \otimes_h \Z \to \Y \otimes_h \W$ is a complete quotient map (i.e. $\otimes_h$ is projective).  
\end{enumerate}
\end{theorem}

We wish to consider to what extent the statements in the above theorem are true in the categories of unital operator spaces and operator systems. Before addressing this, we recall one more important property of the Haagerup tensor product.

\begin{theorem}
    Let $\X$ and $\Y$ be operator spaces and let $\A$ and $\B$ be unital C*-algebras. Suppose that $\pi:\X \to \A$ and $\rho: \Y \to \B$ are complete isometries. Let $a * b$ denote the product of $a \in \A$ and $b \in \A$ inside the free product $\A *_1 \B$ amalgamated over the unit. Define $[\pi * \rho] (x \otimes y) := \pi(x) * \rho(y)$ and extend $\pi*\rho$ to all of $\X \otimes_h \Y$ by linearity. Then $\pi * \rho: \X \otimes_h \Y \to \A *_1 \B$ is a complete isometry.
\end{theorem}

In the above theorem, note that one obtains a complete isometry $\pi * \rho$ regardless of the choice of complete isometries $\pi$ and $\rho$. If the C*-algebra $\A *_1 \B$ is replaced with the non-amalgamated free product $\A * \B$, the result still holds \cite{CESCBMultilinear}. The above statement with amalgamated free products is due to Pisier and its proof employs a dilation trick to substitute the free product with the amalgamated free product (See \cite{PisierKirchbergSimpleProof}). We have the following consequence of the previous theorem (c.f. \cite[Theorem 3.4]{BlecherMagajna10}).

\begin{corollary}
    Suppose that $\V$ and $\W$ are unital operator spaces with units $e\in \V$ and $f\in \W$. Then $\V \otimes_h \W$ is a unital operator space with unit $e\otimes f$.
\end{corollary}

\begin{proof}
    Let $\pi: \V \to \bB(\H)$ and $\rho: \W \to \bB(\K)$ be unital complete order embeddings. Then $\pi * \rho: \V \otimes_h \W \to \bB(\H) *_1 \bB(\K)$ is a unital complete isometry. So $\V \otimes_h \W$ is unitally completely isometric to a unital operator space.
\end{proof}

In particular, if $\S$ and $\T$ are operator systems, then $\S \otimes_h \T$ is a unital operator space. However, since the operator space Haagerup tensor product $\S \otimes_h \T$ is not commuting, there is no compatible way to define an anti-linear involution on $\S \otimes_h \T$, so that generally $\S \otimes_h \T$ is not an operator system.

\section{Haagerup tensor product for unital operator spaces} \label{s:haagerup-tensor}

We begin by continuing the discussion, initiated at the end of the previous section, concerning the Haagerup tensor product of unital operator spaces and operator systems. Our goal here is two-fold. First, the Haagerup tensor product of operator spaces will serve as a key example of an abstract operator space with a hidden multiplicative structure that will be instructive throughout the paper. Second, the properties of Haagerup tensor prodcuts in the category of unital operator spaces have not yet been treated in the literature; so we outline several fundamental properties that are of independent interest. We will be especially interested in the C*-covers of these tensor products which will provide important examples in later sections.

Suppose $\S$ and $\T$ are operator systems. As suggested in Section \ref{s:prelim}, it is rarely the case that $\S \otimes_h \T$ is itself an operator system. This can be seen by considering the following.

\begin{proposition} \label{prop: H-tensor embeds in free prodcut}
    Let $\V_1,\V_2$ be unital operator spaces. Let $\phi_i: \V_i \to C^*(\V_i)$ denote a unital complete isometry into some C*-cover $C^*(\V_i)$ for $i=1,2$. Then $\phi_1 \otimes \phi_2: \V_1 \otimes_h \V_2 \to C^*(\V_1) *_1 C^*(\V_2)$ defined by $[\phi_1 \otimes \phi_2] (x \otimes y) = \phi_1(x) \phi_2(y)$ is a unital complete isometry satisfying $C^*(\phi_1(\V_1) \cdot \phi_2 (\V_2)) = C^*(\V_1) *_1 C^*(\V_2)$. That is, $C^*(\V_1) *_1 C^*(\V_2)$ is a C*-cover for $\V_1 \otimes_h \V_2$.
\end{proposition}

Since the linear span of elements $\{ \ a \cdot b \ | \ a\in \A, \ b\in \B \ \}$ in a free product $\A *_1 \B$ of C*-algebras is necessarily not closed under involution, the range of the unital complete isometry $\phi_1 \otimes \phi_2$ will not be self-adjoint even when $\V_1$ and $\V_2$ are both self-adjoint. This is true even when $\V_1=\V_2$ and $\A:=C^*(\V_1)=C^*(\V_2)$, since the free product $\A *_1 \A$ treats the ``left-hand'' copy of $\A$ and the ``right-hand'' copy of $\A$ as distinct C*-algebras. From this family of C*-covers we obtain the following universal property.

\begin{proposition}
    Let $\phi_i: \V_i \to \bB(\H)$ be unital completely contrative maps for $i=1,2$. Then the map $\phi_1 \otimes_h \phi_2: \V_1 \otimes_h \V_2 \to  \bB(\H)$, defined by $[\phi_1 \otimes_h \phi_2](a \otimes b) = \phi_1(a) \phi_2(b)$ on elementary tensors and extended linearly, is a unital complete contraction.
\end{proposition}

\begin{proof}
    By the universal property of $C^*_u(\V_i)$, there exist unique $*$-homomorphisms $\rho_i: C^*_u(\V_i) \to \bB(\H)$ extending $\phi_i$. Hence we have a $*$-homomorphism $\rho_1 * \rho_2: C^*_u(\V_1) *_1 C^*_u(\V_2) \to \bB(\H)$. Since $\V_1 \otimes_h \V_2 \subseteq C^*_u(\V_1) *_1 C^*_u(\V_2)$ completely isometrically, the restriction of $\rho_1 * \rho_2$ to $\V_1 \otimes_h \V_2$ is the desired map.
\end{proof}

Similarly, we can obtain unital maps between Haagerup tensor products of unital spaces given unital maps between the composite subspaces.

\begin{proposition}
    Let $\phi_i: \V_i \to \W_i$ be unital completely contrative maps for $i=1,2$. Then the map $\phi_1 \otimes_h \phi_2: \V_1 \otimes_h \V_2 \to  \W_1 \otimes_h \W_2$ (defined by $[\phi_1 \otimes_h \phi_2](a \otimes b) = \phi_1(a) \otimes \phi_2(b)$ on elementary tensors and extended linearly) is a unital complete contraction. Moreover, if each map $\phi_i$ is completely isometric, then $\phi_1 \otimes_h \phi_2$ is a unital complete isometry.
\end{proposition}

\begin{proof}
    Let $C^*(\W_i)$ be covers for the unital operator spaces $\W_i$ and let $C^*_u(\V_i)$ be the universal covers for the unital operator spaces $\V_i$. Then we may extend the range of $\phi_i$ to obtain $\phi_i:\V_i \to C^*(\W_i)$. By the universal property of the universal covers, there exist unique $*$-homomorphisms $\tilde{\phi}_i: C_u^*(\V_i) \to C^*(\W_i)$ extending $\phi_i$. By the universal properties of free products, there exists an associated $*$-homomorphism $\tilde{\phi}_1 * \tilde{\phi}_2: C^*_u(\V_1) *_1 C^*_u(\V_2) \to C^*(\W_1) *_1 C^*(\W_2)$. The restriction of this $*$-homomorphism to $\V_1 \otimes_h \V_2$ is the desired map $\phi_1 \otimes_h \phi_2$.

    For the final statement, we recall that whenever $\phi_1$ and $\phi_2$ are completely isometric, then $\phi_1 \otimes \phi_2$ is a complete isometry of operator spaces by the injectivity of the Haagerup tensor product (c.f. \cite[Theorem 17.4]{paulsen2002completely}). So $\phi_1 \otimes \phi_2$ is a unital complete isometry in this case.
\end{proof}

\begin{corollary}
    The Haagerup tensor product is injective in the category of unital operator spaces; i.e. if $\V_1 \subseteq \W_1$ and $\V_2 \subseteq \W_2$ are completely iosmetric inclusions of unital operator spaces, then $\V_1 \otimes_h \V_2 \subseteq \W_1 \otimes_h \W_2$ is a completely isometric inclusion of unital operator spaces.
\end{corollary}

We have seen that $\V_1 \otimes_h \V_2$ embeds completely isometrically into $\B_1 *_1 \B_2$ when $\B_1$ and $\B_2$ are C*-covers of $\V_1$ and $\V_2$ respectively. It is natural to wonder if all C*-covers have this form. The next result shows that this is not the case.

\begin{proposition} \label{p:not-prod-cov}
    Suppose that $\V_1$ and $\V_2$ are non-trivial unital operator spaces. Let $i: \V_1 \otimes_h \V_2 \to C^*_u(\V_1 \otimes_h \V_2)$ be the canonical embedding. If $a \in \V_1$ and $b \in \V_2$ are non-scalar, then $i(a \otimes b) \neq i(a \otimes I) i(I \otimes b)$. Hence the canonical $*$-homomorphism $\rho: C^*_u(\V_1 \otimes_h \V_2) \to C^*_u(\V_1) *_1 C^*_u(\V_2)$ extending the identity map on $\V_1 \otimes_h \V_2$ is not a $*$-isomorphism. 
\end{proposition}

{
\begin{proof}
    First, suppose $\W$ is a non-trivial unital operator space, and $x \in \W$ is non-scalar. We claim there exist ucc maps $\psi_1, \psi_2: \W \to \mathbb{C}$ such that $\psi_1(x) \neq \psi_2(x)$. To see this, it suffices to show that there exist states $\psi_1, \psi_2: \W + \W^* \to \mathbb{C}$ such that either $\psi_1(\Re(x)) \neq \psi_2(\Re(x))$ or $\psi_1(\Im(x)) \neq \psi_2(\Im(x))$, since $\psi_i(\Re(x)) = \Re(\psi_i(x))$, $\psi_i(\Im(x)) = \Im(\psi_i(x))$, and since every ucc map on $\V_i$ extends uniquely to a state on $\W + \W^*$. Now if $x$ is non-scalar, then either $\Re(x)$ or $\Im(x)$ is non-scalar in $\W + \W^*$ and hence this self-adjoint operator will have real numbers $r_1 \neq r_2$ in its spectrum in $C^*_e(\W) = C_e^*(\W+\W^*)$. By functional calculus and Krein's extension theorem, for each point $\lambda$ in the spectrum of a self-adjoint operator in $C^*_e(\W)$ there exists a state on $C^*_e(\W)$ mapping the operator to $\lambda$. Hence, the ucc maps $\psi_1,\psi_2$ exist.

    By the above considerations, and by taking convex combinations of states if necessary, there exist states $\psi_1, \psi_2: \V_1 \to \mathbb{C}$ and $\phi_1, \phi_2: \V_2 \to \mathbb{C}$ such that $\psi_1(a) = \lambda_1, \psi_2(a) = \lambda_2 , \phi_1(b) = \mu_1$, $\phi_2(b) = \mu_2 $, and 
    \[ \lambda_1 \mu_1 + \lambda_2 \mu_2 \neq \frac{1}{2}(\lambda_1 + \lambda_2)(\mu_1 + \mu_2). \]
    Define $\psi: \V_1 \to M_2$ by $\psi = \psi_1 \oplus \psi_2$ and $\phi: \V_2 \to M_2$ by $\phi = \phi_1 \oplus \phi_2$. Let 
    
    \[ P = \frac{1}{2} \begin{pmatrix} 1 & 1 \\ 1 & 1 \end{pmatrix}. \] 
    
    Then, since $P$ is a rank one projection, the map $\theta: x \mapsto P(\psi \otimes \phi)(x)P$ is a ucp map from $\V_1 \otimes_h \V_2$ to $\mathbb{C}$ (identifying $\mathbb{C}$ with the range of the rank one projection $P$). However, $\theta(a \otimes b) \neq \theta(a \otimes I) \theta(I \otimes b)$, so that $\theta$ does not extend to a $*$-homomorphism on $C^*_u(\V_1) *_1 C^*_u(\V_2)$. On the other hand $\theta$ does extend to a $*$-homomorphism from $C^*_u(\V_1 \otimes_h \V_2) \to \mathbb{C}$, so the conclusion follows.
\end{proof}
}

{
\begin{remark} \label{rmk: commuting universal cover does not commmute}
    \emph{By a similar arguments as in the proof of Proposition \ref{p:not-prod-cov}, we see that that there exists a $*$-homomorphism $\theta: C^*_u(\V_1 \otimes_c \V_2) \to \mathbb{C}$ such that $\theta(a \otimes b) \neq \theta(a \otimes I)\theta(I \otimes b)$ for some $a \in \V_1$ and $b \in \V_2$. This is because the maps $\psi$ and $\phi$ constructed in the above proof have commuting range, while their compression by $P$ does not factor as a product of maps with commuting range. It follows that the natural quotient map $C^*_u(\V_1 \otimes_c \V_2) \rightarrow C^*_u(\V_1) \otimes_{max} C^*_u(\V_2)$ is not injective whenever $\V_1$ and $\V_2$ are non-trivial operator systems.}
\end{remark}}

Despite the above proposition, there is still a close relationship between $C^*_u(\V_1 \otimes_h \V_2)$ and $C^*_u(\V_1) *_1 C^*_u(\V_2)$.

\begin{proposition}
    The C*-cover $C^*_u(\V_1) *_1 C^*_u(\V_2)$ is a retract of the universal cover $C^*_u(\V_1 \otimes_h \V_2)$, i.e. there exists an isometric unital $*$-homomorphism $i:C^*_u(\V_1) *_1 C^*_u(\V_2) \to  C^*_u(\V_1 \otimes_h \V_2)$ and a unital surjective $*$-homomorphism $\rho: C^*_u(\V_1 \otimes_h \V_2) \to C^*_u(\V_1) *_1 C^*_u(\V_2)$ such that $\rho \circ i(x) = x$ for all $x \in C^*_u(\V_1) *_1 C^*_u(\V_2)$.
\end{proposition}

\begin{proof}
The existence of $\rho$ follows from the universal property of $C^*_u$, and the fact that $\V_1 \otimes_h \V_2$ is unitally completely isometrically embedded in $C^*_u(\V_1) *_1 C^*_u(\V_2)$. On the other hand, again by the universal property of the maximal C*-algebra the unital complete isometries $\iota_j : \V_j \rightarrow \V_1 \otimes_h \V_2 \subseteq C^*_u(\V_1 \otimes_h \V_2)$ extend to $*$-homomorphisms $\iota_j : C^*_u(\V_j) \rightarrow C^*_u(\V_1 \otimes_h \V_2)$, $j=1,2$, which then in turn extend to the desired map $\iota = \iota_1 * \iota_2 : C^*_u(\V_1) *_1 C^*_u(\V_2) \rightarrow C^*_u(\V_1 \otimes_h \V_2)$ which satisfies $\rho \circ \iota (x) = x$ for every $x\in C^*_u(\V_1) *_1 C^*_u(\V_2)$.
\end{proof}

On the other hand, the C*-envelope is better behaved with respect to products.

\begin{proposition}
    Let $j: \V_1 \otimes_h \V_2 \to C^*_e(\V_1 \otimes_h \V_2)$ denote the canonical inclusion. Then $j(a \otimes b) = j(a \otimes I) j(I \otimes b)$. If we let $C^*(\V_i)$ denote the C*-algebra generated by $j(\V_i)$ inside $C^*_e(\V_1 \otimes_h \V_2)$, then canonically we have $C^*(\V_i) \cong C^*_e(\V_i)$.
\end{proposition}

\begin{proof}
    Since $C^*_e(\V_1) *_1 C^*_e(\V_2)$ is a C*-cover, there exists a $*$-homomorphism $\rho: C^*_e(\V_1) *_1 C^*_e(\V_2) \to C^*_e(\V_1 \otimes_h \V_2)$ such that $\rho(a \otimes b) = j(a \otimes b)$. If we let $k: \V_1 \otimes_h \V_2 \to C^*_e(\V_1) *_1 C^*_e(\V_2)$ denote the embedding, we have $k(a \otimes b) = k(a \otimes I) k(I \otimes b)$. Hence $j(a \otimes b) = \rho(k(a \otimes I) k(I \otimes b)) = (\rho \circ k)(a \otimes I) (\rho \circ k)(I \otimes b) = j(a \otimes I) j(I \otimes b)$. Furthermore, the restriction of $\rho$ to $C^*_e(\V_i) \subseteq C^*_e(\V_1) *_1 C^*_e(\V_2)$ defines a surjective $*$-homomorphism from $C_e^*(\V_i)$ to $C^*(\V_i)$ mapping $k(x)$ to $j(x)$ for every $x \in V_i$. It follows that $C^*(\V_i) \cong C^*_e(\V_i)$.
\end{proof}

The above results reveal some interesting details about maximal ucc maps on $\V_1 \otimes_h \V_2$. Recall that a ucc map $\phi: \V \to \bB(\H)$ is maximal if for any ucc dilation $\psi : \V \rightarrow \bB(\K)$ of $\phi$ we must have that $\psi = \phi \oplus \phi'$ for some ucc map $\phi'$. Moreover, by work of Dritschel and McCullough \cite{DritschelMccullough05} (going back to Muhly--Solel \cite{MuhlySolel1998} and Agler \cite{Agler88}), we know that $\phi$ is maximal if and only if it has the unique extension property, i.e. $\phi$ extends uniquely among ucc maps to a $*$-homomorphism on any C*-cover for $\V$.

\begin{proposition}
    Let $\phi: \V_1 \otimes_h \V_2 \to \bB(\H)$ be a ucc map. Then $\phi$ is maximal if and only if its restrictions to $\V_1$ and $\V_2$ are both maximal.
\end{proposition}

\begin{proof}
    First suppose that $\phi$ is maximal, and let $\phi_i$ be the restriction of $\phi$ to $\V_i$ for $i=1,2$. We will show that $\phi_i$ is maximal for each $i=1,2$. Since $\phi$ is maximal, it extends to a $*$-homomorphism $\rho$ on $C^*_e(\V_1) *_1 C^*_e(\V_2)$, and $\rho$ is unique among ucc extensions of $\phi$. Then $\rho$ decomposes uniquely as $\rho = \rho_1 * \rho_2$ for $*$-homomorphisms $\rho_i: C^*_e(\V_i) \to \bB(\H)$, and $\rho_i$ extends the restriction $\phi_i$ of $\phi$ to $\V_i$ to $*$-representations. In particular, $\phi$ satisfies the partial multiplication property $\phi(v_1 * v_2) = \phi_1(v_1) \phi_2(v_2)$. Now, if $\tau_i$ was a ucc extension of $\phi_i$ to $C^*_e(\V_i)$ for each $i=1,2$, then by Boca's extension theorem \cite[Theorem 3.1]{BOCA1991251} there exists a ucc map $\tau$ on $C^*_e(\V_1) *_1 C^*_e(\V_2)$ whose restriction to $\V_i$ is $\phi_i$, and satisfies the partial multiplication property $\tau(v_1 * v_2) = \phi_1(v_1) \phi_2(v_2) = \rho(v_1 * v_2) = \phi(v_1\otimes v_2)$. Hence, we see that $\tau$ is an extension of $\phi$, and therefore $\tau = \rho$ and is multiplicative. Thus, the restrictions $\rho_i$ of $\rho$ to $C^*_e(\V_i)$ are also multiplicative, and we deduce that $\phi_i$ are each maximal.

    On the other hand, suppose the restrictions $\phi_i$ of $\phi$ to $\V_i$ have the unique extension property. Then these maps extend to $*$-homomorphisms on $C^*(\V_i) \subseteq C^*_e(\V_1 \otimes_h \V_2)$ which are unique among ucc extensions with the same domain. Now let $\rho$ be any ucc extension of $\phi$ to $C^*_e(\V_1 \otimes_h \V_2)$. Then the restriction of $\rho$ to $C^*(\V_i)$ is a ucc extension of $\phi_i$, and is therefore a $*$-homomorphism extending $\phi_i$. Let $\hat{\rho}: C^*_e(\V_1 \otimes_h \V_2) \to \bB(\K)$ be the Stinespring dilation of $\rho$, which is a unital $*$-homomorphism. If $\hat{\rho}_i$ are the restrictions of $\hat{\rho}$ to $C^*(\V_i)$ for $i=1,2$, then $\hat{\rho}_i|_{\V_i}$ is a dilation of the maximal map $\phi_i$. But this implies that $\hat{\rho}_i|_{\V_i}$ is a trivial dilation, i.e. $\hat{\rho}_i|_{\V_i} = \phi_i \oplus \psi_i$ for some ucc map $\psi_i: \V_i \to \bB(\K \ominus \H)$. Since $\hat{\rho}$ is multiplicative, and since $C^*_e(\V_1 \otimes_h \V_2)$ is the C*-algebra generated by $C^*(\V_1)$ and $C^*(\V_2)$, it follows that $\hat{\rho} = \rho \oplus \tau$ for some $*$-homomorphism $\tau$, so that $\rho$ is multiplicative. Thus, $\phi$ has the unique extension property, and is therefore maximal.
\end{proof}

\begin{corollary}
    Let $\V_1$ and $\V_2$ be separable unital operator spaces. Then 
    $$C^*_e(\V_1 \otimes_h \V_2) \cong C^*_e(\V_1) *_1 C^*_e(\V_2).$$
\end{corollary}

\begin{proof}
    To prove this result we compare universal properties. First note that by the co-universal property of the C*-envelope, since $(C^*_e(\V_1) *_1 C^*_e(\V_2),j)$ is a C*-cover for $\V_1 \otimes_h \V_2$, there is a surjective $*$-homomorphism $q : C^*_e(\V_1) *_1 C^*_e(\V_2) \rightarrow C^*_e(\V_1 \otimes_h \V_2)$ such that the map $q\circ j = \kappa$ is the embedding of $\V_1\otimes_h \V_2$ into $C^*_e(\V_1 \otimes_h \V_2)$.

    Recall that the norm of a polynomial $p\in C^*_e(\V_1) *_1 C^*_e(\V_2)$ is the supremum of norms $\|(\pi_1 * \pi_2)(p)\|$ ranging over all pairs of $*$-homorphisms $\pi_i : C^*_e(\V_i) \rightarrow \bB(\H)$ where $\H$ is a separable Hilbert space. By work of Dritschel-McCullough \cite{DritschelMccullough05} (see also \cite{Arveson08}) we know that $C^*_e(\V_1 \otimes_h \V_2)$ is the universal C*-algebra with respect to ucc maps $\phi : \V_1 \otimes_h \V_2 \rightarrow \bB(\H)$ which are maximal, so that by the previous proposition $\phi = \phi_1 \otimes \phi_2$ such that $\phi_i$ is maximal for $i=1,2$. Thus, to show that $q$ is injective, it will suffice to show that the norm of $p$ is equal to the supremum of norms of norms $\|(\pi_1 * \pi_2)(p)\|$ ranging over all pairs of $*$-homorphisms $\pi_i : C^*_e(\V_i) \rightarrow \bB(\H)$ such that $\pi_i|_{\V_i}$ is maximal for $i=1,2$.
    
    Take an injective $*$-representation $\pi = \pi_1 * \pi_2$ of the free product $C^*_e(\V_1) *_1 C^*_e(\V_2)$ of infinite multiplicity, into separable Hilbert space. It will suffice to show that $\pi_1$ and $\pi_2$ can be approximated in the point-norm topology by another pair $\pi_1'$ and $\pi_2'$ such that $\pi_i'|_{\V_i}$ is maximal. Indeed, suppose for every $\epsilon>0$, finite $\F_1 \subset C^*_e(\V_1)$ and finite $\F_2 \subset C^*_e(\V_2)$ there exist $\pi_i'$ such that $\|\pi_i(a_i) - \pi_i'(a_i) \| < \epsilon$ for $a_i \in \F_i$. Then  
    for every finite family of polynomials $\F  \subseteq C^*_e(\V_1) *_1 C^*_e(\V_2)$ and every $\delta>0$ we may choose $\epsilon$ depending on $\delta$ and the polynomials in $\F$ so that there exist $\pi_i'$ for $i=1,2$ as above satisfying $\| (\pi_1* \pi_2)(p) - (\pi_1' * \pi_2')(p)\| < \delta$ for every $p\in \F$. This can be achieved by letting $\F_i$ be the factors of monomials comprising the polynomials in $\F$ that belong to $C^*_e(\V_i)$ for each $i=1,2$ and applying standard polynomial analysis arguments.

    Every injective $*$-representation $\pi : C^*_e(\V_i) \rightarrow \bB(\H)$ of infinite multiplicity into separable $\H$ for some unital operator space $\V_i$ can be approximated in the point-norm topology by a $*$-representation $\pi' : C^*_e(\V_i) \rightarrow \bB(\H)$ such that $\pi'|_{\V_i}$ is maximal. Indeed, by Dritschel--McCullough there is some injective, infinite multiplicity $*$-representation $\pi' : C^*_e(\V_i) \rightarrow \bB(\K)$ such that $\pi'|_{\V_i}$ is maximal. Since all representations considered are of infinite multiplicity, their quotients into the Calkin algebra are still injective, and therefore by Voiculescu's theorem (see \cite[Corollory II.5.6]{DavidsonCstarBook}) there exists a sequence $U_n : \H \rightarrow \K$ such that $\Ad_{U_n} \circ \pi'$ converges in the point norm to $\pi$.
\end{proof}

\color{black}

\section{Products of unital operator spaces} \label{s:product-structure}

In the previous section, we saw that Haagerup tensor product of two unital operator spaces $(\S,e)$ and $(\T,f)$ is itself a unital operator space. Moreover, there exists a complete isometry $\pi: \S \otimes_h \T \to \bB(\H)$ such that $\pi(s \otimes t) = \pi(s) \pi(t)$ for all $s \in \S$ and $t \in \T$. It follows that $\V = \S \otimes_h \T$ admits a partially-defined multiplication $m: \D \to \V$, where \[ \D = (\bC e \times \V) \cup (\V \times \bC f) \cup (\S \times \T), \] defined by $m(\lambda e \otimes f, x) = m(x, \lambda e \otimes f) = \lambda x$ and $m(s,t) = s \otimes t$ for all $\lambda \in \bC$, $x \in \V$, $s \in \S$, and $t \in \T$. The map $\pi$ described above satisfies $\pi(m(a,b)) = \pi(a)\pi(b)$ for all $(a,b) \in \D$.

In this section, we wish to abstractly characterize this situation. Specifically, whenever $\V$ is a untial operator space, $\D \subseteq \V \times \V$, and $m: \D \to \V$ is a map, we wish to characterize exactly when there exists a unital complete isometry $\pi: \V \to \bB(\H)$ such that $\pi(m(a,b)) = \pi(a)\pi(b)$ for all $(a,b) \in \D$.

Let $\V$ be a unital operator space with unit $e \in \V$. We call a subset $\D \subseteq \V \times \V$ a \textit{domain} if it satisfies the following conditions:
\begin{enumerate}
    \item $\mathbb{C}e \times \V \subseteq \D$ and $\V \times \mathbb{C} e \subseteq \D$.
    \item For each $x \in \V$, there exists a unital subspace $\R_x \subseteq \V$ such that $\bC x \times \R_x \subseteq \D$.
    \item For each $y \in \V$, there exists a unital subspace $\L_y \subseteq \V$ such that $\L_y \times \bC y \subseteq \D$.
\end{enumerate}

Given a domain $\D$, a map $m: \D \to \V$ is called \textit{bilinear} if for every $x,y,z \in \V$ and $\lambda \in \bC$, $y,z \in \L_x$ implies $m(y+\lambda z,x) = m(y,x) + \lambda m(z,x)$ and $y,z \in \R_x$ implies $m(x,y + \lambda z) = m(x,y) + \lambda m(x,z)$. A pair of matrices $A = (a_{ij}) \in M_{n,k}(\V)$ and $B = (b_{ij}) \in M_{k,m}(\V)$ are called a \textit{valid pair} provided that $(a_{x,y}, b_{y,z}) \in \D$ for all indices $x,y,z$. Given a valid pair $(A,B)$, we define
\[ A \odot_m B := ( \sum_y m(a_{x,y}, b_{y,z}))_{x,z}. \]
A bilinear map $m:\D \to \V$ is \textit{completely contractive} provided that $\|A \odot_m B \| \leq \|A\| \|B\|$ for all valid pairs $A$ and $B$. It is called \textit{unital} if $m(x,e)=m(e,x)=x$ for all $x \in \V$. A \textit{partial product} on $\V$ consists of a pair $(\D,m)$, where $\D$ is a domain and $m$ is a unital bilinear map.

To illustrate the above definitions, we consider the following natural constructions and examples which will make appearances throughout the paper.

\begin{construction}[Trivial Product] \label{const: trivial product}
    \emph{Let $\V$ be a unital operator space with unit $e$. Since multiplication by the identity leaves $\V$ invariant, we may take $\D = (\bC e \times \V) \cup (\V \times \bC e)$ and define $m(\lambda e,x) = m(x, \lambda e) = \lambda$ for all $\lambda \in \bC$. By considering unital complete isometries of $\V$ into $\bB(\H)$, we see that $m$ is a completely contractive partial product.}
\end{construction}

\begin{construction}[Unital operator algebra]
    \emph{Let $\A$ be a (possibly non-self adjoint) unital operator algebra. Then we let $\D = \A \times \A$ and $m(x,y) = x \cdot y$ be the given product. Then $m$ is a completely contractive partial product.}
\end{construction}

\begin{construction}[Haagerup tensor product] \label{const: H tensor}
    \emph{Let $\S$ and $\T$ be unital operator spaces with units $I_{\S}$ and $I_{\T}$ respectively, and let $\V = \S \otimes_h \T$. Then the unital embedding $\pi$ of $\V$ into $C^*_u(\S) *_1 C^*_u(\T)$ satisfies $\pi(s \otimes I_{\T}) \pi(I_{\S} \otimes t) = s \cdot t = \pi(s \otimes t)$ for all $s \in \S$ and $t \in \T$. Identifying $\S$ with $\S \otimes \bC I_{\T}$, $\T$ with $\bC I_{\S} \otimes \T$, and $\bC$ with $\bC(I_{\S}\otimes I_{\T})$, we may set $\D = (\mathbb{C} \times \V) \cup (\V \times \mathbb{C}) \cup (\S \times \T)$ and define $m(s,t) = s \otimes t$ for all $s \in \S$ and $t \in \T$. By Proposition \ref{prop: H-tensor embeds in free prodcut}, $m$ is a completely contractive partial product.}
\end{construction}

\begin{construction}[Commuting tensor product] \label{const: c tensor}
    \emph{Let $\S$ and $\T$ be operator systems with units $I_{\S}$ and $I_{\T}$ respectively, and let $\V = \S \otimes_c \T$. Then the unital embedding $\pi$ of $\V$ into $C^*_u(\S) \otimes_{max} C^*_u(\T)$ satisfies $\pi(s \otimes I_{\T}) \pi(I_{\S} \otimes t) = \pi(I_{\S} \otimes t) \pi(s \otimes I_{\T}) = s \otimes t = \pi(s \otimes t)$ for all $s \in \S$ and $t \in \T$. Identifying $\S$ with $\S \otimes I_{\T}$, $\T$ with $I_{\S} \otimes \T$, and and $\bC$ with $\bC(I_{\S}\otimes I_{\T})$, we may set $\D = (\bC \times \V) \cup (\V \times \bC) \cup (\S \times \T) \cup (\T \times \S)$ and $m(s,t) = m(t,s) = s \otimes t$ for all $s \in \S$ and $t \in \T$. Then $m$ is a completely contractive partial product.}
\end{construction}

\begin{example}[Free unitary system] \label{ex: free unitary}
    \emph{Let $n \in \mathbb{N}$ with $n \geq 2$ and let $\S_n$ denote the linear span of the set $\{I, u_1, \dots, u_n, u_1^*, \dots, u_n^*\} \subseteq C^*(\mathbb{F}_n)$, where $C^*(\mathbb{F}_n)$ denotes the universal C*-algebra of the free group on $n$ generators $\mathbb{F}_n = \langle u_1, \dots, u_n \rangle$. In this operator system, we have the relation $u_i u_i^* = u_i^* u_i = I$ for each $i =1,\dots,n$. Since these are the only non-trivial products which leave $\S_n$ invariant, it is easy to see that $\R_{u_i} =\L_{u_i} = \text{span} \{e, u_i^*\}$ and $\R_{u_i^*} = \L_{u_i^*} = \text{span} \{e, u_i\}$ for each $i=1,\dots,n$. In this case, we may take \[ \D = (\S_n \times \bC) \cup (\bC \times \S_n) \cup \left( \cup_{i=1}^n (\L_{u_i} \times \R_{u_i^*}) \cup (\L_{u_i^*} \times \R_{u_i}) \right) \]
    and define a completely contractive partial product $m$ such that $m(u_i,u_i^*)=m(u_i^*,u_i)=I$ for all $i=1,\dots,n$.}
\end{example}

\begin{example}[Cuntz system] \label{ex: Cuntz system}
    \emph{Let $s_1, s_2, \dots, s_n$ be the generators of the Cuntz algebra $\O_n$. For each $i=1,2,\dots,n$, let $p_i = s_i s_i^*$. Let $\V_n = \text{span} \{I, s_1, \dots, s_n, s_1^*, \dots, s_n^*, p_1, \dots, p_n\}$. Then $\V_n$ is an operator system. By \cite{zhangPaulsen16cuntz}, $C^*_e(\V_n) \cong \O_n$. We may define a partial multiplication on $\V_n$ by letting $m(s_i^*, s_i) = e$ and $m(s_i, s_i^*) = p_i$ for each $i=1,2,\dots,n$, and extending $m$ to a partial product on the domain
    \[ \D = (\V_n \times \mathbb{C}) \cup (\mathbb{C} \times \V_n) \cup \left( \cup_{i=1}^n (\L_{s_i} \times \R_{s_i^*}) \cup (\L_{s_i^*} \times \R_{s_i}) \right) \] where $\R_{s_i} = \L_{s_i} = \text{span} \{I, s_i\}$ and $\R_{s_i^*} = \L_{s_i^*} = \text{span} \{I, s_i\}$. Then $m$ is a completely contractive partial product.}
\end{example}

At the end of this section, we will provide examples of partial products which are not completely contractive. To justify these, we first need to prove some results about completely contractive partial products.

\subsection{Unital operator spaces} \label{subsec: unital op sp}

Before considering the general case, we will focus on the special case of a \textit{trivial partial product}, i.e. we consider the situation where $\D = (\mathbb{C}e \times \V) \cup (\V \times \mathbb{C}e)$ and $m$ is defined by setting $m(\lambda e, x) = m(x, \lambda e) = x$ for all $\lambda \in \mathbb{C}$ and $x \in \V$. For a trivial partial product, a pair of matrices $A =(a_{ij})$ and $B=(b_{ij})$ over $\V$ is a valid pair if and only if either $a_{ik}$ or $b_{kj}$ is a multiple of the unit $e$ for all indices $i,j$, and $k$. We recall the following result from \cite[Theorem 2.3]{blecher2011metric}.

\begin{theorem}[Blecher-Neal] \label{thm: Blecher-Neal unit}
    Let $\V$ be an operator space and $e \in \V$ a unit vector. Then there exists a complete isometry $\pi: \V \to \bB(\H)$ such that $\pi(e)=I$ if and only if for every $x \in M_n(\V)$,
    \[ \| \begin{pmatrix} I_n \otimes e & x \end{pmatrix} \| = \|\begin{pmatrix} I_n \otimes e \\ x \end{pmatrix} \| = \sqrt{1 + \|x\|^2}. \]
    Moreover, it suffices for this equation to hold whenever $\|x\|=1$.
\end{theorem}

The forward direction of Theorem \ref{thm: Blecher-Neal unit} is easily checked by considering a unital complete isometry into $\bB(\H)$. The other direction makes use of the theory of multiplier algebras, which we will consider in the next subsection. We remark that 
\[ \| \begin{pmatrix} I_n \otimes e & x \end{pmatrix} \| = \|\begin{pmatrix} I_n \otimes e \\ x \end{pmatrix} \| \leq \sqrt{\|e\|^2 + \|x\|^2} \]
is valid in any operator space (and checked by considering complete isometries $\pi:\V \to \bB(\H)$).

\begin{theorem} \label{thm: Unit characterization}
    Let $\V$ be an operator space and $e \in \V$ an element such that $\|e\| = 1$. The following statements are equivalent:
    \begin{enumerate}
        \item There exists a linear complete isometry $\pi: \V \to \bB(\H)$ such that $\pi(e)=I$.
        \item For every valid pair $(A,B)$, $\|A \odot_m B \| \leq \|A\| \|B\|$ where $m$ is the trivial partial product (i.e. the trivial partial product is completely contractive).
        \item For every $x, y \in M_n(\V)$,
        \[ \|x+y\| \leq \| \begin{pmatrix} x & I_n \otimes e \end{pmatrix} \| \| \begin{pmatrix} I_n \otimes e \\ y \end{pmatrix} \| \]
    \end{enumerate}
\end{theorem}

\begin{proof}
    The implications $(1) \implies (2) \implies (3)$ are trivial. We show that $(3) \implies (1)$. By Theorem \ref{thm: Blecher-Neal unit}, (1) is equivalent to the equation 
    \begin{equation} \| \begin{pmatrix} I_n \otimes e & x \end{pmatrix} \| = \|\begin{pmatrix} I_n \otimes e \\ x \end{pmatrix} \| = \sqrt{2} \label{eqn: BN} \end{equation}
    whenever $\|x\| = 1$. Thus, it suffices to show that (3) implies Equation \ref{eqn: BN}. To that end, suppose that $x \in M_n(\V)$ and $\|x\| = 1$. Then
    \[ 2 = \|x + x\| \leq \| \begin{pmatrix} x & I_n \otimes e \end{pmatrix} \| \| \begin{pmatrix} I_n \otimes e \\ x \end{pmatrix} \| \]
    by statement (3). Hence
    \[ \| \begin{pmatrix} x & I_n \otimes e \end{pmatrix} \| \geq \frac{2}{\| \begin{pmatrix} I_n \otimes e \\ x \end{pmatrix} \|} \geq \frac{2}{\sqrt{1 + \|x\|^2}} = \sqrt{2} \]
    since 
    \[ \| \begin{pmatrix} a \\  b \end{pmatrix} \| \leq \sqrt{\|a\|^2 + \|b\|^2} \]
    for any $a,b \in M_n(\V)$. Because
    \[ \| \begin{pmatrix} x & I_n \otimes e \end{pmatrix} \| \leq \sqrt{ \|x\|^2 + \|I_n \otimes e\|^2} = \sqrt{2}, \]
    it follows that $\| \begin{pmatrix} x & I_n \otimes e \end{pmatrix} \| = \| \begin{pmatrix} I_n \otimes e & x \end{pmatrix} \| = \sqrt{2}$. The other equality is similar.
\end{proof}

\begin{remark}
    \emph{Suppose that $\S$ is an operator system with unit $e$. Then it is well-known that any complete isometry $\pi:\ S \to \bB(\H)$ such that $\pi(e)=I$ is necessarily $*$-preserving and a complete order embedding (c.f. \cite[Proposition 3.5]{paulsen2002completely}). It is also possible to identify which elements of $\S$ are self-adjoint and which are positive by examining the matrix norms. For example, it was shown in \cite{blecher2011metric} that if $\S$ is an operator space with unit $e$, then $x \in M_n(\S)$ is self-adjoint if and only if $\|I_n \otimes e + itx\|_n^2 = 1 + t^2\|x\|^2$ for all $t \in \mathbb{R}$, and is positive if and only if it is self-adjoint and $\| \|x\|e - x\| \leq \|x\|$.}
\end{remark}

\subsection{General partial products}

We now consider the general situation of potentially non-trivial partial products. We will make use of some results concerning multiplier algebras. Given a unital operator space $\V$, a linear map $\lambda: \V \to \V$ is called a \textit{left multiplier} if there exists $x \in I(\V)$ such that $\lambda(y) = xy$ for all $y \in \V$. Similarly, a map $\rho: \V \to \V$ is a \textit{right multiplier} if there exists $z \in I(\V)$ such that $\rho(y) = yz$ for all $y \in \V$. Since $\V$ is unital, it follows that whenever $\lambda(y) = xy$ is a left multiplier (or $\rho(y)=yz$ is a right multiplier), then $x \in \V$ since $\lambda(e)=x$ (and $z \in V$ since $\rho(e)=z$). This following is a combination of Theorem 16.12 \cite{paulsen2002completely} and remarks found after \cite[Cor 1.8]{BlecherMultipliers} and \cite[p.192]{BlecherOperatorAlgebras}.

\begin{theorem}[Blecher-Effros-Zarikian] \label{thm: BEZ}
    Let $\V$ be a unital operator space. Suppose that $\varphi: \V \to \V$ is completely contractive, and define $\tau: C_2(\V) \to C_2(\V)$ by
    \[ \tau \left( \begin{pmatrix} x \\ y \end{pmatrix} \right) = \begin{pmatrix} \varphi(x) \\ y \end{pmatrix}. \]
    If $\tau$ is completely contractive, then $\tau$ is a left multiplier, i.e. for every $x \in \V$, $\varphi(x) = \varphi(I) \cdot x$, where the product is taken in $C^*_e(\V)$.
    
    Similarly, if the map $\gamma: R_2(\V) \to R_2(\V)$ given by 
    \[ \gamma \left( \begin{pmatrix} x & y \end{pmatrix} \right) = \begin{pmatrix} \varphi(x) & y \end{pmatrix} \]
    is completely contractive, then for every $x \in \V$, $\varphi(x) = x \cdot \varphi(I)$ where the product is taken in $C^*_e(\V)$.
\end{theorem}

The next result can be found in \cite[Proposition 15.10]{paulsen2002completely}.

\begin{lemma} \label{lem: C*-subalg of injective envelope}
    Suppose $1 \in \A \subseteq \S \subseteq \bB(\H)$ with $\A$ a C*-algebra and $\S$ an operator system. Then the inclusion of $\A$ into $I(\S)$, where the latter is imbued with its Choi--Effros product, is an injective $*$-homomorphism.
\end{lemma}

We will apply the Blecher-Effros-Zarikian Theorem to show that $m(a,b) = a \cdot b$ in $C^*_e(\V)$. Unfortunately, we cannot apply the Theorem to the operator space $V$ directly. This is because ``left multiplication by $a$'' and ``right multiplication by $b$'' are only defined as maps from $\R_a \to \V$ and $\L_b \to \V$, respectively. The proof of Theorem \ref{thm: BEZ} (e.g. the arguments in \cite[Chapter 15]{paulsen2002completely}) does not lend itself to obvious modification for our setting. However, it suffices to consider a different operator system for fixed $(a,b) \in \D$, namely
\[ \R := \begin{pmatrix} \L_b & \V \\ 0 & \R_a \end{pmatrix} \]
regarded as a subsystem of $M_2(\V)$. We can then define $\varphi_a, \psi_b: \R \to \R$ by
\[ \varphi_a \begin{pmatrix} x & y \\ 0 & z \end{pmatrix} = \begin{pmatrix} 0 & m(a,z) \\ 0 & 0 \end{pmatrix} \quad \text{and} \quad \psi_b \begin{pmatrix} x & y \\ 0 & z \end{pmatrix} = \begin{pmatrix} 0 & m(x,b) \\ 0 & 0 \end{pmatrix} \]
mimicking the matrix multiplications
\[ \begin{pmatrix} 0 & a \\ 0 & 0 \end{pmatrix} \begin{pmatrix} x & y \\ 0 &  z \end{pmatrix} = \begin{pmatrix} 0 & az \\ 0 & 0 \end{pmatrix} \quad \text{and} \quad \begin{pmatrix} x & y \\ 0 & z \end{pmatrix} \begin{pmatrix} 0 & b \\ 0 &  0 \end{pmatrix} = \begin{pmatrix} 0 & xb \\ 0 & 0 \end{pmatrix}. \]
We will need to understand $C^*_e(\R)$ in some detail to apply Theorem \ref{thm: BEZ}.

\begin{proposition} \label{prop: Structure of C*_e(V) }
    Let $\V$ be a unital operator space, let $\S$ and $\T$ be unital subspaces, and let
    \[ \R = \begin{pmatrix} \S & \V \\ 0 & \T \end{pmatrix} \]
    regarded as a unital subspace of $M_2(\V)$. Then $C^*_e(\R) \cong M_2(C^*_e(\V))$ via the $*$-isomorphism identifying $\R$ with the subspace of $M_2(\V)$ inside $M_2(C^*_e(\V))$.
\end{proposition}

\begin{proof}
    First recall that $C^*_e(\R)$ is the C*-subalgebra generated by $\R$ as a subset of its injective envelope $I(\R)$, where the latter is imbued with the Choi--Effros product. Since $\S$ and $\T$ are unital and the inclusion of $\R$ into $I(\R)$ fixes the $2 \times 2$ diagonal scalar matrices (by Lemma \ref{lem: C*-subalg of injective envelope}), $I(\R)$ decomposes as a $2 \times 2$ matrix of operator spaces with $\S$ contained in the upper left corner, $\T$ in the lower right, and $\V$ in the upper right (c.f. the proof of \cite[Theorem 15.12]{paulsen2002completely}).

    Since $\S$, $\T$, and $\V$ are all unital, we see that $M_2(\bC) \subseteq \R + \R^*$. Since $C^*_e(\R) = C^*_e(\R + \R^*)$, we may apply Lemma \ref{lem: C*-subalg of injective envelope} again and the fact that $C^*_e(\R)$ is the C*-algebra generated by $\R$ inside its injective envelope to conclude that 
    \[ \begin{pmatrix} a & b \\ c & d \end{pmatrix} \mapsto \begin{pmatrix} aI_{\S} & b U \\ c U^* & d I_{\T} \end{pmatrix} \]
    is an injective $*$-homomorphism, where $I_{\S}$ denotes the unit in $\S$, $I_{\T}$ denotes the unit in $\T$ and $U$ denotes the unit in $\V$. It follows that
    \[ \begin{pmatrix} 0 & 0 \\ U^* & 0 \end{pmatrix} \begin{pmatrix} 0 & U \\ 0 & 0 \end{pmatrix} = \begin{pmatrix} 0 & 0 \\ 0 & U^*U \end{pmatrix} = \begin{pmatrix} 0 & 0 \\ 0 & I_{\T} \end{pmatrix} \] and
    \[ \begin{pmatrix} 0 & U \\ 0 & 0 \end{pmatrix} \begin{pmatrix} 0 & 0 \\ U^* & 0 \end{pmatrix} = \begin{pmatrix}UU^* & 0 \\0 & 0 \end{pmatrix} = \begin{pmatrix} I_{\S} & 0 \\ 0 & 0 \end{pmatrix} \]
    Thus, we see that $U$ is a unitary, and we may assume $I(\R) \subseteq \bB(\H \oplus \H)$ for some Hilbert space $\H$ and that $I_{\S} = I_{\T} = I_{\H}$. Since $I(\R)$ is a C*-algebra with this $2\times 2$ block structure, it follows that the upper left corner $\A$ and lower right corner $\B$ are C*-algebras and that the upper right corner $\M$ is an $\A-\B$ bimodule.

    Finally, we want to show that $\A = \B = C^*_e(\V)$ and that (without loss of generality) $U = I_{\S} = I_{\T} = I_{\V}$. Since $M_2(C^*_e(\V))$ is a C*-cover for $\R$, there exists a surjective $*$-homomorphism $\pi: M_2(C^*_e(\V)) \to C^*_e(\R)$ which restricts to a unital complete isometry on $\R$. For any $x \in \V$, we have
    \[ \pi \begin{pmatrix} x & 0 \\ 0 & 0 \end{pmatrix} = \pi \begin{pmatrix} 0 & x \\ 0 & 0 \end{pmatrix} \pi \begin{pmatrix} 0 & 0 \\ I & 0 \end{pmatrix} = \begin{pmatrix} 0 & \hat{x} \\ 0 & 0 \end{pmatrix} \begin{pmatrix} 0 & 0 \\ U^* & 0 \end{pmatrix}. \] So $\pi(x) = \hat{x} U^*$ where $\hat{x}$ denotes the image of $x$ in $\M$ and $\pi(x)$ denotes the image of $x$ in $\A$. Restricting $\pi$ to the upper left corner of $M_2(C^*_e(\V))$, we see that $\A$ is a quotient of $C^*_e(\V)$. Since $\A$ contains an isometric copy of $\V$, by the co-universal property of the C*-envelope, we get that $\A \cong C^*_e(\V)$. Similarly, we obtain that $\B \cong C^*_e(\V)$. Finally, the restriction of $\pi$ to the upper right corner of $M_2(C^*_e(\V))$ shows that $\M$ is a ternary quotient of $C^*_e(\V) \cong T^*_e(\V)$. Since $\M$ is a ternary cover of $\V$, we get that $\M \cong T^*_e(\V) \cong C^*_e(\V)$ where the isomorphism $T^*_e(\V) \cong C^*_e(\V)$ identifies the untitary $U$ with $I_{\V}$.
\end{proof}

\begin{theorem} \label{thm: pp rep in C*-env}
    Let $\V$ be a unital operator space, $\D \subseteq \V \times \V$ a domain, and $m: \D \to \V$ a completely contractive partial product. Let $j: \V \to C^*_e(\V)$ be the embedding of $\V$ into its C*-envelope. Then for fixed $(a,b) \in \D$ we have that $j(m(a,b)) = j(a)j(b)$.
\end{theorem}

\begin{proof}
    Let $\S = \L_b$ and $\T = \R_a$. Let
    \[ \R = \begin{pmatrix} \S & \V \\ 0 & \T \end{pmatrix} \]
    regarded as a subsystem of $M_2(\V)$. For each $a,a' \in \S$, $b \in \T$, and $x \in \V$, define
    \[ \varphi_a \begin{pmatrix} a' & x \\ 0 & b \end{pmatrix} = \begin{pmatrix} 0 & m(a,b) \\ 0 & 0 \end{pmatrix}. \]
    If $\|a\| > 1$, then rescale $a$ so that $\|a\| = 1$. Then we claim that
    \[ \tau_a \left( \begin{pmatrix} X \\ Y \end{pmatrix} \right) = \begin{pmatrix} \varphi_a(X) \\ Y \end{pmatrix} \]
    is completely contractive. To see this, let $X, Y \in M_n(\R)$. Then up to canonical shuffle,
    \[ X = \begin{pmatrix} a' & x \\ 0 & b  \end{pmatrix} \quad \text{and} \quad Y = \quad \begin{pmatrix} c & y \\ 0 & d  \end{pmatrix} \]
    for $a', c \in M_n(\S)$, $b,d \in M_n(\T)$, and $x,y \in M_n(\V)$. Then clearly
    \[ \| \begin{pmatrix} X' \\ Y \end{pmatrix} \| \leq \| \begin{pmatrix} X \\ Y \end{pmatrix} \| \quad \text{where} \quad X' = \begin{pmatrix} 0 & 0 \\ 0 & b \end{pmatrix}.  \]
    Now
    \begin{eqnarray}
        \| \tau_a^{(n)} \left( \begin{pmatrix} X \\ Y \end{pmatrix} \right) \| = \| \begin{pmatrix} \varphi_a^{(n)}(X) \\ Y \end{pmatrix} \| & = & \| \begin{pmatrix} 0 & a \odot_m b \\ 0 & 0 \\ c & y \\ 0 & d \end{pmatrix} \| \nonumber \\
        & = & \| \begin{pmatrix} 0 & a & 0 & 0 \\ 0 & 0 & 0 & 0 \\ 0 & 0 & I & 0 \\ 0 & 0 & 0 & I \end{pmatrix} \odot_m  \begin{pmatrix} 0 & 0 \\ 0 & b \\ c & y \\ 0 & d \end{pmatrix} \| \nonumber \\
        & \leq & \| \begin{pmatrix} X' \\ Y \end{pmatrix} \| \leq \| \begin{pmatrix} X \\ Y \end{pmatrix} \|. \nonumber
    \end{eqnarray}
    By Theorem \ref{thm: BEZ} we get that $\varphi_a(x) = \varphi_a(I) \cdot x$ in $C^*_e(\R)$ for all $x \in \R$. By Proposition \ref{prop: Structure of C*_e(V) }, we see that
    \[ \varphi_a(I) = \begin{pmatrix} 0 & j(a) \\ 0 & 0 \end{pmatrix} \quad \text{and so} \quad \begin{pmatrix} 0 & m(a,b) \\ 0 & 0 \end{pmatrix} = \varphi_a \begin{pmatrix} 0 & 0 \\ 0 & j(b) \end{pmatrix} = \begin{pmatrix} 0 & j(a)j(b) \\ 0 & 0 \end{pmatrix}. \]
    It follows that $j(m(a,b)) = j(a) j(b)$ in $C^*_e(\V)$, since $C^*_e(\R) = M_2(C^*_e(\V))$.
\end{proof}

\begin{corollary} \label{cor: abstract product characterization}
    Let $\V$ be an operator space and $e \in \V$ a unit vector. Then $(\D,m)$ is a completely contractive partial product on $(\V,e)$ if and only if there exists a Hilbert space $H$ and a complete isometry $\pi: \V \to \bB(\H)$ such that $\pi(m(a,b))=\pi(a)\pi(b)$ for all $a,b \in \V$ and $\pi(e)=I$.
\end{corollary}

\begin{remark} \emph{Recall that when $\V$ is a unital operator space, the operator system $\V + \V^*$ is unique. Hence every unital operator space has a unique partially defined adjoint. In the case when $(a,b) \in \D$ and $a^*, b^*, m(a,b)^* \in V$, then since $m(a,b)^* = (ab)^* = b^*a^*$ in $C^*_e(\V)$, we can extend the domain $\D$ so that $(b^*,a^*) \in D$ and $m(b^*,a^*) = m(a,b)^*$. In particular, if $\V$ is an operator system, then whenever $(\D,m)$ is a partial multiplication and $(a,b), (b^*,a^*) \in \D$, it is necessarily true that $m(a,b)^* = m(b^*,a^*)$. Thus, if $(a,b) \in \D$ but $(b^*,a^*) \notin \D$, we can extend $\D$ to include $(b^*,a^*)$ by setting $m(b^*,a^*) = m(a,b)^*$ and the multiplication will still be completely contractive. In general, even if $(a,b) \in \D$ implies $\pi(a)\pi(b) \in \pi(\V)$ for some complete order embedding $\pi: \V \to \bB(\H)$, there may be other pairs $(x,y) \in (\V \times \V) \setminus \D$ such that $\pi(x)\pi(y) \in \pi(\V)$ as well. In other words, the set $\D$ need not be a maximal set of pairs whose product is also in $\V$.}
\end{remark}

We conclude this section by exhibiting some instances of partial products that are not completely contractive.

\begin{construction}[Maximal tensor product] \label{ex. Max tensor}
    \emph{Let $\S$ and $\T$ be operator systems with units $I_{\S}$ and $I_{\T}$ respectively, and consider $\V = \S \otimes_{max} \T$. Identifying $\S$ with $\S \otimes I_{\T}$ and $\T$ with $I_{\S} \otimes \T$ inside $\V$, we obtain a partial product satisfying $m(s,t)=m(t,s)=s \otimes t$ for all $s \in \S$ and $t \in \T$ (as defined for $\S \otimes_c \T$ in Construction \ref{const: c tensor}). Then $m$ is not completely contractive in general. This is because, by \cite{KAVRUK2011}, there exist operator systems $\S$ and $\T$ such that $\S \otimes_c \T \neq \S \otimes_{max} \T$. However, if $m$ were completely contractive, then by Corollary \ref{cor: abstract product characterization} there would exist a unital complete isometry $\pi: \S \otimes_{max} \T \to \bB(\H)$ such that $\pi(a \otimes b) = \pi(a) \pi(b) = \pi(b) \pi(a)$ for all $a \in \S$ and $b \in \T$. The existence of such a map would imply that $\S \otimes_{max} \T = \S \otimes_c \T$, a contradiction.}
\end{construction}

\begin{example}[Anticommutator] \label{ex: anticommutator}
    \emph{Let Let $\S \subseteq \bB(\H)$ be an operator system and consider $\V = \S^2 = \{ xy : x,y \in S \}$. Then $\V$ is an operator system. Let $\D = (\V \times \mathbb{C} I_{\H}) \cup (\mathbb{C}I_{\H} \times \V) \cup (\S \times \S)$. Then we can define a partial product $m: \D \to \V$ by $m(x,y) = \{x,y\} := \frac{1}{2}(xy + yx)$. We claim that this partial product is not completely contractive in general. If it were, then by Corollary \ref{cor: abstract product characterization} there would exist unital complete isometry $\pi: \V \to \bB(\H)$ such that $\pi(\{x,y\}) = \pi(x)\pi(y) = \pi(y) \pi(x)$ for all $x,y \in \S$. Consider the case when $\S = M_2$. Then it is known that $\V = \S^2 = \S$, but the product $m(x,y) = \{x,y\}$ is not associative. Hence, there exist $x_1, x_2, x_3 \in S$ such that $m(x_1,x_2), m(x_2,x_3) \in S$, but $m(m(x_1,x_2),x_3) \neq m(x_1, m(x_2,x_3))$. However, \[ \pi(m(m(x_1,x_2),x_3)) = \pi(x_1)\pi(x_2)\pi(x_3) = \pi(m(x_1,m(x_2,x_3))). \] 
    But this is not possible since $\pi$ is a complete isometry. So this partial product is not completely contractive in the case when $\S = M_2$.}
\end{example}

\section{Universal product C*-cover and product quotients} \label{s:universal-covers-quotients-prod}

In the previous section, we saw that if $\V$ is a unital operator space with partial product $m: \D \to \V$, then $j(m(a,b))= j(a)j(b)$, where $j: \V \to C^*_e(\V)$ is the canonical embedding. However, there may be other C*-covers which respect a given partial product. For example, if $\V = \S \otimes_h \T$, then the partial product defined by $m(s,t) = s \otimes_h t$ for all $s \in \S$ and $t \in \T$ is respected in both C*-covers $C^*_u(\S) *_1 C^*_u(\T)$ and $C^*_e(\S) *_1 C^*_e(\T) = C^*_e(\S \otimes_h \T)$. However, it is straitforward to check, using the universal properties of the universal C*-cover and the free product, that if $i: \V \to \A = C^*(\V)$ is a C*-cover such that $i(m(a,b))=i(a)i(b)$ for all $(a,b) \in \D$, then a map $\pi: C^*_u(\S) *_1 C^*_u(\T) \rightarrow \A$ satisfying $\pi(s \otimes t) = i(s)i(t)$ exists and is ucp. This shows that the C*-cover $C^*_u(\S) *_1 C^*_u(\T)$ is the ``largest'' C*-cover respecting the partial product on $\V$. 

In this section, we will show that every operator space with a completely contractive partial product admits a universal C*-cover which respects partial products. We will also show that the quotient of such an operator space by the kernel of a ucp map respecting the partial product also possesses a completely contractive partial product, and that this quotient space satisfies natural universal properties.

\begin{definition}
    Let $\V$ and $\W$ be unital operator spaces with partial products $m_{\V}: \D_{\V} \to \V$ and $m_{\W}: \D_{\W} \to \W$, respectively. A linear map $\varphi: \V \to \W$ is called a \textit{product map} if for any $(a,b) \in \D_{\V}$ we have $(\varphi(a),\varphi(b)) \in \D_{\W}$, and
    $\varphi(m_V(a,b)) = m_{\W}(\varphi(a), \varphi(b))$ for all $(a,b) \in \D_{\V}$. By a \textit{product C*-cover} for $\V$, we mean a pair $(\A,j)$ where $\A$ is a C*-algebra, and $j: \V \to \A$ is a unital completely isometric product map with $\D_{\A} = \A \times \A$ and $C^*(j(\V)) = \A$.
\end{definition}

By Theorem \ref{thm: pp rep in C*-env}, whenever $m:\D \to \V$ is a completely contractive product map, then $C^*_e(\V)$ is a product C*-cover for $\V$. 

\begin{theorem} \label{thm: c-star-m universal property}
    Let $\V$ be a unital operator space with completely contractive partial product $m:\D \to \V$. Then there exists a product C*-cover $(C^*_m(\V),i)$ satisfying the following universal property: for every product C*-cover $(\A,j)$ there exists a $*$-homomorphism $\pi: C^*_m(\V) \to \A$ satisfying $\pi(i(x)) = j(x)$ for all $x \in \V$.
\end{theorem}

\begin{proof}
    Let $(C^*_u(\V), k)$ denote the universal C*-cover for $\V$ with embedding $k: \V \to C^*_u(\V)$. Let $\I_m$ denote the ideal generated by the set $\{ k(x)k(y) - k(m(x,y)) : (x,y) \in \D \}$. We claim that $C^*_m(\V) := C^*_u(\V) / \I_m$ is a product C*-cover for $\V$ with embedding $j: \V \to C^*_m(\V)$ defined by $j(x) = k(x) + \I_m$. To see this, let $(\A,i)$ be a product C*-cover (e.g. we could take $\A = C^*_e(\V)$). Consider the quotient map $\pi: C^*_u(\V) \to \A$ satisfying $\pi(k(x))=\pi(i(x))$ (which exists and is unique by the universal property of $C^*_u(\V)$). Since $\A$ is a product C*-cover, $\pi(k(x)k(y)-k(m(x,y))) = i(x)i(y)-i(m(x,y)) = 0$ for all $(x,y) \in \D$. It follows that $\I_m \subseteq \ker(\pi)$. It follows that the map $j(x) = k(x) + \I_m \mapsto i(x)$ extends to a $*$-homomorphism $\pi': C^*_m(\V) \to \A$. If we let $\rho: C^*_u(\V) \to C^*_m(\V) = C^*_u(\V) / \I_m$ denote the quotient map, then we have that the restrictions $\rho: k(\V) \to j(\V)$ and $\pi': j(\V) \to i(\V)$ are all unital complete isometries. Therefore $j: \V \to C^*_m(\V)$ is a unital complete isometry. By the definition of $\I_m$, it is clear that $j:\V \to C^*_m(\V)$ is a product map. The universal property follows from the argument above, since there is at most one $*$-homomorphism $\pi': C^*_m(\V) \to \A$ satisfying $\pi'(j(x))=i(x)$ for a given product C*-cover $(\A,i)$.
\end{proof}

By the universal properties of the C*-covers $C^*_u(\V)$ and $C^*_e(\V)$, there exist $*$-homomorphisms $\rho: C^*_u(\V) \to C^*_m(\V)$ and $\pi: C^*_m(\V) \to C^*_e(\V)$ each extending the identity map on $\V$. In the following examples, we will see cases where we can write $C^*_m(\V)$ somewhat explicitly. While $C^*_m(\V)$ is generally distinct from $C^*_u(\V)$ and $C^*_e(\V)$, we see from these examples that it is possible for $C^*_m(\V) = C^*_u(\V)$ or for $C^*_m(\V) = C^*_e(\V)$.

\begin{construction}[Unital operator spaces]
    \emph{Suppose $\V$ is a unital operator space and $m$ is the trival product. Then $C^*_m(\V) = C^*_u(\V)$. This is because any C*-cover is a product C*-cover for $m$, so there is a unique $*$-homomorphism $\pi: C^*_m(\V) \to C^*_u(\V)$ extending the identity on $\V$. By the universal property of $C^*_u(\V)$, the identity on $V$ extends to a unique $*$-homomorphism $\rho: C^*_u(\V) \to C^*_m(\V)$ since $C^*_m(\V)$ is a C*-cover. It follows that $\rho = \pi^{-1}$ and thus $C^*_u(\V) = C^*_m(\V)$.}
\end{construction}

\begin{construction}[Unital operator algebra]
    \emph{Let $\A$ be a unital operator algebra. Then $C^*_m(\A) = C^*_{u-alg}(\A)$ where $C^*_{u-alg}(\A)$ denotes the universal C*-cover in the category or unital operator algebras. Like the previous example, this follows from the fact that $C^*_{u-alg}(\A)$ is a product C*-cover, inducing a $*$-homomorphism $\pi: C^*_m(\A) \to C^*_{u-alg}(\A)$; and by the universal property of $C^*_{u-alg}(\A)$, inducing a $*$-homomorphism $\rho: C^*_{u-alg}(\A) \to C^*_m(\A)$. The existence of $\rho$ is due to the observation that $C^*_m(\A)$ contains $\A$ as a unital subalgebra.}
\end{construction}

\begin{construction}[Haagerup tensor product]
    \emph{Let $\S$ and $\T$ be unital operator spaces and define $m$ on $\V = \S \otimes_h \T$ by $m(s,t) = s \otimes t$ for all $s \in \S$ and $t \in \T$. Suppose that $j: \V \to \A = C^*(\V)$ is a product C*-cover. Then $j(s \otimes t) = j(s)j(t)$ for all $s \in \S$ and $t \in \T$, and $\A$ is generated by $j(\S)$ and $j(\T)$. By the universal properties of $C^*_u(\S)$ and $C^*_u(\T)$, there exists unital $*$-homomorphisms $\rho_1: C^*_u(\S) \to C^*(j(\S))$ and $\rho_2: C^*_u(\T) \to C^*(j(\T))$. By the universal properties of the free product, there exists a $*$-homomorphism $\pi = \rho_1 * \rho_2: C^*_u(\S) *_1 C^*_u(\T) \to A$ extending $\rho_1$ and $\rho_2$. Evidently $\pi$ is surjective. Since $C^*_u(\S) *_1 C^*_u(\T)$ is a C*-cover, we conclude that $C^*_m(\V) = C^*_u(\S) *_1 C^*_u(\T)$. By Proposition \ref{p:not-prod-cov}, we see that the canonical surjections $C^*_u(\V) \rightarrow C^*_m(\V)$ and $C^*_m(\V) \rightarrow C^*_e(\V)$ are not injective in general.}
\end{construction}

\begin{construction}[Commuting tensor product]
    \emph{Let $\S$ and $\T$ be operator systems and define $m$ on $\V = \S \otimes_{max} \T$ by $m(s,t) = m(t,s)= s \otimes t$ for all $s \in \S$ and $t \in \T$. Suppose that $j: \V \to \A = C^*(\V)$ is a product C*-cover. Then $j(s \otimes t) = j(s)j(t)=j(t)j(s)$ for all $s \in \S$ and $t \in \T$, and $\A$ is generated by $j(\S)$ and $j(\T)$. By the universal properties of $C^*_u(\S)$ and $C^*_u(\T)$, there exists unital $*$-homomorphisms $\rho_1: C^*_u(\S) \to C^*(j(\S))$ and $\rho_2: C^*_u(\T) \to C^*(j(\T))$. By the universal properties of the maximal C*-tensor product, there exists a $*$-homomorphism $\pi = \rho_1 \otimes \rho_2: C^*_u(\S) \otimes_{max} C^*_u(\T) \to \A$ extending $\rho_1$ and $\rho_2$. Evidently $\pi$ is surjective. Since $C^*_u(\S) \otimes_{max} C^*_u(\T)$ is a C*-cover, we conclude that $C^*_m(\V) \cong C^*_u(\S) \otimes_{max} C^*_u(\T)$. By Remark \ref{rmk: commuting universal cover does not commmute}, we see that $C^*_u(\V) \neq C^*_m(\V)$ and $C^*_m(\V) \neq C^*_e(\V)$ in this case as well.}
\end{construction}

\begin{construction}[Minimal tensor product]
    \emph{Let $\S$ and $\T$ be operator systems and define $m$ on $\V = \S \otimes_{min} \T$ by $m(s,t) = m(t,s) = s \otimes t$ for all $s \in \S$ and $t \in \T$. Suppose that $j: \V \to \A = C^*(\V)$ is a product C*-cover. Let $\A_1 = C^*(j(\S))$ and $\A_2 = C^*(j(\T))$. Then $\A_1 \otimes_{min} \A_2$ is also a product C*-cover for $\S \otimes_{min} \T$, by the injectivity of the minimal tensor product \cite{KAVRUK2011}. Let $\B = C^*_m(\S \otimes_{min} \T)$, and let $\B_1 = C^*(\S) \subseteq \B$ and $\B_2 = C^*(\T) \subseteq \B$. If $\rho: \B \to \A$ is the canonical $*$-homomorphism extending the identity map on $\S \otimes_{min} \T$, then the restriction of $\rho$ to $\B_i$ induces a $*$-homomorphism $\rho_i: \B_i \to \A_i$ for $i=1,2$. Since $C^*_u(\S) \otimes_{min} C^*_u(\T)$ is a C*-cover for $\S \otimes_{min} \T$, we conclude that $\B_1 = C^*_u(\S)$ and $\B_2 = C^*_u(\T)$. It follows that $C^*_m(\S \otimes_{min} \T) = C^*_u(\S) \otimes_{\tau} C^*_u(\T)$, where $\otimes_{\tau}$ is some C*-tensor product. We do not know if $C^*_m(\S \otimes_{min} \T) = C^*_u(\S) \otimes_{min} C^*_u(\T)$ in general, even though $C^*_u(\S) \otimes_{min} C^*_u(\T)$ is always a product C*-cover. For instance, it may be the case that $\S \otimes_{min} \T = \S \otimes_c \T$ (e.g. if $\S$ is (min,c)-nuclear), in which case $C^*_u(\S) \otimes_{max} C^*_u(\T)$ is a product C*-cover for $\S \otimes_{min} \T$ by \cite[Theorem 6.4]{KAVRUK2011} and therefore $C^*_m(\S \otimes_{min} \T) \cong C^*_u(\S) \otimes_{max} C^*_u(\T)$. We do not know if $\S \otimes_{min} \T = \S \otimes_c \T$ implies $C^*_u(\S) \otimes_{max} C^*_u(\T) \cong C^*_u(\T) \otimes_{min} C^*_u(\T)$. However, we can conclude that if $C^*_m(\S \otimes_{min} \T) \neq C^*_u(\S) \otimes_{max} C^*_u(\T)$, then $\S \otimes_{min} \T \neq \S \otimes_c \T$.}
\end{construction}

\begin{example}[Free unitary system]
    \emph{Let $\S_n = \text{span} \{e, u_1, \dots, u_n, u_1^*, \dots, u_n^*\} \subseteq C^*(\mathbb{F}_n)$ as in Example \ref{ex: free unitary}. Let $j: \S_n \to \A = C^*(\S_n)$ be a product C*-cover. Then because $j(u_i)j(u_i)^* = j(u_i)^*j(u_i) = I$, we get that $j(u_i)$ is a unitary for each $i=1,\dots,n$. It follows that the mapping $u_i \mapsto j(u_i)$ extends to a $*$-homomorphism $\pi: C^*(\mathbb{F}_n) \to \A$. Since $C^*(\mathbb{F}_n)$ is a product C*-cover, we deduce that $C^*_m(S_n) = C^*(\mathbb{F}_n)$. By \cite{FKPTgroups2014}, we have $C^*_e(\S_n) = C^*(\mathbb{F}_n)$ as well, so in our case $C^*_m(\S_n) \cong C^*_e(\S_n)$. On the other hand, the image of $u_i$ in $C^*_u(\S_n)$ is not a unitary since its spectrum must be the entire unit ball of $\mathbb{C}$. Therefore $C^*_u(\S_n) \neq C^*_m(\S_n)$.}
\end{example}

\begin{example}[Cuntz system]
    \emph{Let $\V_n = \text{span} \{I, s_1, \dots, s_n, s_1^*, \dots, s_n^*, p_1, \dots, p_n\}$ as in Example \ref{ex: Cuntz system}. Then $C^*_e(\V_n) \cong \O_n$. Using the universal properties of the Cuntz algebra, then (as in the previous example) we have $C^*_m(\V_n) = \O_n = C^*_e(\V_n)$. However, $C^*_u(\V_n) \neq C^*_e(\V_n)$, since the image of each non-zero projection $p_i$ in $C^*_u(\V_n)$ has spectrum equal to the closed unit interval $[0,1]$. So $C^*_u(\V_n) \neq C^*_m(\V_n)$.}
\end{example}

\begin{definition}
    Let $V$ be a unital operator space with a completely contractive partial product $m: \D \to \V$. A subspace $\K \subseteq \V$ is called a \textit{product kernel} if it is the kernel of a ucc product map.
\end{definition}

Recall that if $\V$ is an operator system and $\K \subseteq \V$ is the kernel of a ucp map, then there exists a matrix ordering on $\V / \K$ making it an operator system with unit $e+\K$. This operator system satisfies the following universal property: if $\phi: \V \to \bB(\H)$ is a ucp map such that $\K \subseteq \ker(\phi)$, then the map defined by $x + \K \mapsto \phi(x)$ is a well-defined ucp map (i.e. $\phi$ factors through the quotient $\V / \K$). We will show that when $\V$ is a unital operator space and $\K$ is a \emph{product} kernel, then there exists a product unital operator space structure on $\V / \K$ satisfying an analogous universal property for product maps. We call $\V / \K$ with this structure a \textit{product quotient} and denote it by $(\V / \K)_m$.

\begin{theorem}
    Let $\V$ be a unital operator space and $m: \D \to \V$ a completely contractive partial product. Then there exists an operator space structure on $\V / \K$ such that the partial product defined by \[ m(x + \K, y + \K) = m(x,y) + \K, \quad (x,y) \in \D \] is completely contractive. Moreover, if $\psi: \V \to \bB(\H)$ is a ucc product map with $\K \subseteq \ker(\psi)$, then the map $x + \K \mapsto \psi(x)$ is a ucc product map, i.e. $\psi$ factors through $(\V / \K)_m$ as a product map.
\end{theorem}

\begin{proof}
    Since $\K = \ker(\varphi)$ where $\varphi: \V \to \bB(\H)$ is a product map, the unique induced $*$-homomorphism $\pi_{\varphi}: C^*_m(\V) \to \bB(\H)$ satisfies $\langle \K \rangle = \ker(\pi_{\varphi})$. It follows that $\langle \K \rangle \cap \V = \K$, and therefore the map sending $\V / \K$ into $C^*_m(\V) / \langle \K \rangle$ given by $x + \K \mapsto j(x) + \langle \K \rangle$ is injective.

    Endow $\V / \K$ with the operator space structure making the inclusion $\V / \K \subseteq C^*_m(\V) / \langle \K \rangle$ completely isometric. Then the partial product defined by $m(x + \K, y + \K) := m(x,y) + \K$ for all $(x,y) \in \D$ is completely contractive, since it agrees with the product in the C*-algebra $C^*_m(\V) / \langle \K \rangle$.

    Now suppose that $\psi: \V \to \bB(\H)$ is a ucc product map and $\K \subseteq \ker(\psi)$. Then $\psi$ extends to a $*$-homomorphism $\pi_{\psi}: C^*_m(\V) \to \bB(\H)$. Since $\K \subseteq \ker(\psi)$, $\langle \K \rangle \subseteq \ker(\pi_{\psi})$. It follows that the restriction $\psi'$ of $\pi_{\psi}$ to $\V / \K$ is a ucc product map, and $\psi'(x + \K) = \pi_{\psi}(x + \langle \K \rangle) = \psi(x)$.
\end{proof}

The following was shown in the proof above, but we state it separately for visibility. Recall that when $S$ is an operator system and $\J \subseteq \S$ is a kernel, then the inclusion of $\S / \J$ into $C^*_u(\S)/\langle \J \rangle$ is a complete isometry \cite{Kavruk}. A similar situation occurs in the case of product quotients.

\begin{corollary} \label{cor: quotient C*-cover}
    Let $\V$ be a unital opreator space, $m:\D \to \V$ a completely contractive partial product, and $\K \subseteq \V$ the kernel of a ucc product map. Then the inclusion of $(\V / \K)_m$ into $C^*_m(\V) / \langle \K \rangle$ is a complete isometry.
\end{corollary}

\begin{remark}
    \emph{As discussed in \cite{Kavruk}, it is non-trivial to determine when a subspace $\K$ of an operator system $\S$ is the kernel of a ucp map. Likewise, it appears to be difficult to determine when a kernel $\K \subseteq \V$ is a product kernel. Although we do not have an intrinsic characterization of such kernels, some natural necessary conditions can be deduced. For example, if $(x,y) \in \D$ and $y \in \K$ or $x\in \K$, then necessarily $m(x,y) \in \K$, i.e. $\K$ behaves partly like an ideal in an algebra. While this property is strong enough to induce a partial product on the quotient $\V / \K$, it is not clear if the induced product is automatically completely contractive.}
\end{remark}

We conclude this section with some applications of product quotients. First, if $\V$ is an operator system and $m$ denotes the trivial product, and if $\K \subseteq \V$ is the kernel of a ucp map then it is necessarily a product kernel with respect to $m$. Moreover, $(\V / \K)_m$ is the operator system quotient, since product maps are just ucp maps in this case. This can also be seen by \cite{Kavruk}, since $C^*_m(\V / \K) \cong C^*_m(\V)/\langle \K \rangle = C^*_u(\V) / \langle \K \rangle$. 

Second, if we denote by $m$ the natural partial product on the Haagerup tensor product $\S \otimes_h \T$ of unital operator spaces $\S$ and $\T$, the following establishes product projectivity of Haagerup tensor product.

\begin{theorem} \label{t:product-proj-haagerup}
    Let $\S$ and $\T$ be unital operator spaces, $\J \subseteq \S$ and $\K \subseteq \T$ kernels, and let $q_1: \S \to \S / \J$ and $q_2: \T \to \T / \K$ denote the complete quotient maps. Let $\L = \J \otimes \T + \S \otimes \K$. Then the product map $q: (\S \otimes_h \T / \L)_m \to (\S / \J) \otimes_h (\T / \K)$, defined by $q(s \otimes t + \L) = (s+\J) \otimes_h (t + \K)$ on elementary tensors and extended linearly, is a completely isometry.
\end{theorem}

\begin{proof}
First, note that the complete quotient maps $q_1: \S \to \S / \J$ and $q_2: \T \to \T / \K$ extend to a ucc map $q_1 \otimes q_2: \S \otimes_h \T \to (\S / \J) \otimes_h (\T / \K)$ with kernel $\L = \J \otimes \T + \S \otimes \K$. Now, since $q_1 \otimes q_2$ is a product map, $\L$ is a product kernel. By Corollary \ref{cor: quotient C*-cover}, we know the inclusion
\[ (\S \otimes_h \T / \L)_m \to C^*_m(\S \otimes_h \T)/\langle \L \rangle \cong C^*_u(\S) *_1 C^*_u(\T) / \langle \L \rangle \]
is a unital complete isometry. Since $C^*_u(\S / \J) \cong C^*_u(\S) / \langle \J \rangle$ and $C^*_u(\T / \K) \cong C^*_u(\T) / \langle \K \rangle$, we have a $*$-homomorphism $\pi: C^*_u(\S) *_1 C^*_u(\T) \to C^*_u(\S / \J) *_1 C^*_u(\T / \K)$, namely $\pi_1 * \pi_2$ where the $\pi_i$'s are quotient maps. Since $\L \subseteq \ker(\pi)$, this induces a $*$-homomorphism from $C^*_m(\S \otimes_h \T) / \langle \L \rangle \to C^*_u(\S / \J) *_1 C^*_u(\T / \K)$. However
\[ C^*_u(\S / \J) *_1 C^*_u(\T / \K) \cong (C^*_u(\S) / \langle \J \rangle) *_1 (C^*_u(\T) / \langle \K \rangle) \cong C^*_u(\T) *_1 C^*_u(\S) / \langle \J + \K \rangle. \]
Because $\langle \L \rangle = \langle \J + \K \rangle$, the proof is concluded.
\end{proof}

\begin{remark}
  \emph{If $\S$ and $\T$ are unital operator spaces with kernel $\J$ and $\K$, respectively, we do not know if the unital operator space quotient satisfies $((\S \otimes_h \T)/(\S \otimes \K + \J \otimes \T))_{usp} \cong (\S / \J) \otimes_h (\T / \K)$ as unital operator spaces. In other words, we do not know if the Haagerup tensor product is projective in the category of unital operator spaces.}  
\end{remark}

\begin{theorem} \label{t:commutig-is-haagerup-prod-quotient}
    let $\S$ and $\T$ be operator systems, and let $\V = \S \otimes_h \T + \T \otimes_h \S$ denote the operator system generated by the image of $\S \otimes_h \T$ in $C^*_u(\S) *_1 C^*_u(\T)$. Let \[ [\S,\T] := \text{span}\{s \otimes t - t \otimes s : s \in \S, t \in \T \}. \] Then
    \[ ((\S \otimes_h \T + \T \otimes_h \S) / [\S,\T])_m \cong \S \otimes_c \T \]
    via the mapping defined by $s \otimes t + t' \otimes s' + [\S,\T] \mapsto s \otimes t + s' \otimes t'$ on elementary tensors and extended linearly.
\end{theorem}

\begin{proof}
First, note that the partial product on $\V$ satisfies $m(s \otimes I, I \otimes t) = s \otimes t$ and $m(t \otimes I, I \otimes s) = t \otimes s$. Since the map $\varphi: \S \otimes_h \T \to \S \otimes_c \T$ given by $\varphi(s \otimes t) = s \otimes t$ is a ucc product map, it extends to a ucp product map $\varphi': \S \otimes_h \T + \T \otimes_h \S \to \S \otimes_c \T$. The kernel of this map is $[\S,\T]$. Consider the embedding 
\[ ((\S \otimes_h \T + \T \otimes_h \S) / [\S,\T])_m \to C^*_m(\S \otimes_h \T + \T \otimes_h \S) / \langle [\S,\T] \rangle \cong C^*_u(\S) *_1 C^*_u(\T) / \langle [\S,\T] \rangle. \] 
Since the images of $\S$ and $\T$ commute in this quotient C*-algebra, it follows that \[ C^*_u(\S) *_1 C^*_u(\T) / \langle [\S,\T] \rangle  = (C^*_u(\S) + \langle [\S,\T] \rangle)(C^*_u(\T) + \langle [\S,\T] \rangle) \] and hence is product of commuting C*-covers. So we have a $*$-homomorphim $$\pi: C^*_u(\S) \otimes_{max} C^*_u(\T) \to C^*_u(\S) *_1 C^*_u(\T) / \langle [\S,\T] \rangle.$$ On the other hand, the map $\rho: C^*_u(\S) *_1 C^*_u(\T) \to C^*_u(\S) \otimes_{max} C^*_u(\T)$ satisfies $\langle [\S,\T] \rangle \subseteq \ker(\rho)$, so we conclude that $C^*_u(\S) *_1 C^*_u(\T) / \langle [\S,\T] \rangle \cong C^*_u(\S) \otimes_{max} C^*_u(\T) \cong C^*_m(\S \otimes_c \T)$. Thus, the proof is concluded.
\end{proof}

\begin{remark}
    \emph{For operator systems $\S$ and $\T$ we do not know whether or not 
    $$
    ((\S \otimes_h \T + \T \otimes_h \S) / [\S,\T])_{osy} \cong \S \otimes_c \T,
    $$
    where $(\V / \K)_{osy}$ denotes the operator system quotient of an operator system $\V$ by a kernel $\K$.}
\end{remark}

\section{Factorization norms and their applications} \label{s:factorization-application}

Let $\V$ be an operator space, $e \in \V$ with $\|e\|=1$. Let $\D \subseteq \V \times \V$ be a domain and $m:\D \to \V$ a partial product (not necessarily completely contractive). In this section, we consider methods to replace the matrix norm on $\V$ with another matrix norm making $m$ completely contractive. With these methods, we find a variety of new intrinsic expressions for the matrix norms on various unital operator spaces and operator systems.

Recall that a \textit{valid pair} consists of a pair of matries $(A,B)$ over $\V$ such that their matrix product $A \odot_m B$ is well defined, i.e. $(A_{ik},B_{kj}) \in \D$ for all indices $i,j,k$. In the following, we want to consider possibly ``invalid'' products of matrices that nonetheless result in matrices with entries in $\V$. To give a simple example, consider $\V = \text{span} \{e, x\}$, $\D = (\V \times \mathbb{C}e) \cup (\mathbb{C}e \times \V)$, and $m$ is the trivial product defined by $m(x,e)=m(e,x)=x$ and $m(e,e)=e$. Then the pair
\[ ( \begin{bmatrix} e & x & -x \end{bmatrix}, \begin{bmatrix} 3x \\ 2x \\ 2x \end{bmatrix}) \]
is invalid, but
\[ \begin{bmatrix} e & x & -x \end{bmatrix} \odot_m \begin{bmatrix} 3x \\ 2x \\ 2x \end{bmatrix} = 3x + 2x^2 - 2x^2 = 3x. \]
We call such tuples \textit{permissible}. 

To rigorously define the term, we need to consider the algebraic embedding of $\V$ into a certain free algebra, denoted $\F_m(\V)$. Let $\F(\V)$ denote the algebraic free algebra over $\V$ (i.e. the Fock algebra over $\V$). Let $\J_m$ denote the ideal of $\F(V)$ generated by elements of the form $a \cdot b - m(a,b)$ where $(a,b) \in \D$. We let $\F_m(\V) := \F(\V) / \J_m$ and write $a \cdot_m b$ (or just $a \cdot b$ when there is no risk of confusion) to denote the product $(a + \J_m) \cdot (b + \J_m)$ in $\F_m(\V)$. We remark that $\F_m(\V)$ is a unital algebra, with unit $e$, and that $a \cdot_m b = m(a,b)$ for all $(a,b) \in \D$. In general, words of length greater than two in $\F(\V)$ can be reduced to elements of $\V$ in $\F_m(\V)$. For instance, if $(a,b) \in \D$ and $(m(a,b),c) \in \D$, then $a \cdot_m b \cdot_m c = m(m(a,b),c)$ in $\F_m(\V)$. We say that the partial product $m$ is \textit{non-degenerate} if the linear embedding of $\V$ into $\F_m(\V)$, given by $v \mapsto v + \J_m$, is injective.

In general, a partial product on a vector space may be degenerate. For example, suppose $\D$ is a domain with $(a,b), (m(a,b),c), (a,m(b,c)) \in \D$. If $m(m(a,b),c) \neq m(a,m(b,c))$, then the partial product will be degenerate, since the image of $m(m(a,b),c) - m(a,m(b,c))$ will be zero in $\F_m(\V)$ (for instance, the partial product in Example \ref{ex: anticommutator} is degenerate for this reason). Rather than algebraically characterizing when a partial product is non-degenerate, we will make use of the following observation.

\begin{lemma}
    Let $\V$ be a vector space with vector $e$ and $m:\D \to \V$ a partial product. Then $m$ is non-degenerate if and only if there exists a unital algebra $\A$ and an injective linear map $\varphi: \V \to \A$ such that $\varphi(e)=I$ and $\varphi(m(a,b)) = \varphi(a) \varphi(b)$ for all $(a,b) \in \D$.
\end{lemma}

\begin{proof}
    Let $\pi_{\varphi}: \F(\V) \to \A$ be the unique homomorphism extending $\varphi$. Since $\varphi(m(a,b))= \varphi(a) \varphi(b)$, we have $\pi_{\varphi}(m(a,b) - a \cdot b) = 0$ for all $(a,b) \in \D$. It follows that $\J_m \subseteq \ker(\pi_{\varphi})$. Thus the map $\pi': \F(\V)/\J_m = \F_m(\V) \to \A$ defined y $\pi'(x + \J_m) = \pi(x)$ is a well-defined homomorphism. Because $\varphi$ is injective, $\pi'(x + \J_m) = 0$ implies that $x = 0$ whenever $x \in \V$. We conclude that the embedding $x \mapsto x + \J_m$ of $\V$ into $\F_m(\V)$ is injective. On the other hand, if $m$ is non-degenerate, then we may take $\A = \F_m(\V)$ and $\varphi: x \mapsto x + \J_m$ is injective.
\end{proof}

Suppose $\V$ is a vector space with unit $e$ and a non-degenerate partial product $m$. We say that a tuple $(X^{(1)}, X^{(2)}, \dots, X^{(n)})$ of matrices over $\V$ is \textit{permissible} if the number of rows of $X^{(k)}$ equals the number of columns of $X^{(k+1)}$ for $k=1,\dots,n-1$ and if the entries of the matrix $(X^{(1)} \odot_m \dots \odot_m X^{(n)})$ defined by
\[ (X^{(1)} \odot_m \dots \odot_m X^{(n)})_{ij} = \sum_{k_1, \dots, k_{n-1}} X_{i,k_1} \cdot_m X_{k_1,k_2} \cdot_m \dots \cdot_m X_{k_{n-1},j} \]
are all in $\V$, where the product is taken in $\F_m(\V)$.

\begin{lemma} \label{l:permis}
    Assume $\V$ is a vector space and $m:\D \to \V$ is a non-degenerate partial product with unit $e \in \V$. Let $A \in M_{n,k}(\V)$ and $B = M_{k,l}(\V)$. If $(A,B)$ is a valid pair, $A = A_1 \odot_m \dots \odot_m A_{N}$ and $B = B_1 \odot_m \dots \odot_m B_{M}$ with $(A_1, \dots, A_N)$ and $(B_1, \dots, B_M)$ permissible, then $(A_1, \dots, A_N, B_1, \dots, B_M)$ is a permissible tuple satisfying \[ A \odot_m B = A_1 \odot_m \dots \odot_m A_N \odot_m B_1 \odot_m \dots \odot_m B_M. \]
\end{lemma}

\begin{proof}
    Non-degeneracy of the product $m$ implies that $\V$ can be identified as a subspace of $\F_m(\V)$, and the rest follows from associativity of the product in $\F_m(\V)$.
\end{proof}

Now assume $\V$ is an operator space and that $e \in \V$ with $\|e\|=1$. Let $\D$ be a domain and $m: \D \to \V$ a non-degenerate partial product.  Given a matrix $A \in M_{n,k}(\V)$, we define
\[ \|A\|_{n,k}^m := \inf \{ \|a_1\| \|a_2\| \dots \|a_n\| : A = a_1 \odot_m \dots \odot_m a_n \} \]
where the infimum is taken over permissible factorizations of $A$ into matrices over $\V$. It is clear that $\|A\|_{n,k}^m \leq \|A \|_{n,k}$ for every $A\in M_{n,k}(\V)$. We will show that, under moderate restrictions, $(\V,e, \{\|\cdot\|_n^m\})$ is a unital operator space and $m$ is a completely contractive unital partial product on $\V$.

Recall that a sequence of maps $\|\cdot\|_n: M_n(\X) \to [0,\infty)$ is an $\mathrm{L}^\infty$-\textit{matrix seminorm} if for all $\alpha, \beta \in M_{n,k}$ and all $A \in M_n(\X)$, $B \in M_k(\X)$, we have
\[ \|\alpha^* A \beta \|_k \leq \|\alpha\| \|\beta\| \|A\|_n \quad \text{and} \quad \|A \oplus B\|_{n+k} = \max(\|A\|, \|B\|). \]

\begin{proposition} \label{prop: factorization seminorm}
    Let $\X$ be an operator space with unit $e$, let $\D$ be a domain and $m : \D \rightarrow \X$ a partial product on $\X$. Then the sequence $\{\|\cdot\|_n^m\}$ defines an $\mathrm{L}^\infty$-matrix seminorm on $\X$. Moreover, if there exists an \emph{injective} unital completely contractive product map $i : \X \rightarrow \bB(\H)$ (with respect to the original operator space structure on $\X$), then $\{\|\cdot\|_n^m\}$ is an $\mathrm{L}^{\infty}$-matrix \emph{norm}. In this case, $i: \X \to \bB(\H)$ is completely contractive with respect to the matrix norm $\{\|\cdot\|_n^m\}$.
\end{proposition}

\begin{proof}
    Let $\alpha, \beta \in M_{n,k}$ and let $A \in M_n(\X)$. Suppose $A$ admits a permissible factorization $A = A_1 \odot A_2 \odot \dots \odot A_N$. Then $\alpha^* A \beta$ admits the factorization $(\alpha^* \otimes e) \odot A_1 \odot A_2 \odot \dots \odot A_N \odot (\beta \otimes e)$ and hence
    \[ \|\alpha^*A\beta\|_n^m \leq \|\alpha\| \|A_1\| \|A_2\| \dots \|A_N\| \|\beta\|. \]
    Taking an infimum over all such factorizations, we see that $\|\alpha^*A\beta\|_n^m \leq \|\alpha\| \|\beta\| \|A\|_n^m$.

    Now let $A \in M_n(\X)$ and $B \in M_k(\X)$. We will demonstrate that $\|A \oplus B\|_{n+k}^m = \max(\|A\|_n^m, \|B\|_k^m)$. Without loss of generality, we may assume  $\|A\|_n^m \geq \|B\|_k^m$. Then for every $\epsilon > 0$ we can find factorizations $A = A_1 \odot \dots \odot A_N$ and $B = B_1 \odot \dots \odot B_M$ such that $\prod_s \|A_s\| \leq \|A\|_n^m + \epsilon$ and $\prod_t \|B_t\| \leq \|B\|_n^m + \epsilon$. If $N \neq M$, we can introduce scalar multiples $I_n \otimes e$ or $I_k \otimes e$ to obtain a new factorization with $N=M$ and leaving $\prod_s \|A_s\|$ and $\prod_t \|B_t\|$ unchanged. By renormalizing the $A_s$'s and $B_t$'s, we can also assume $\|A_s\| \leq (\|A\|_n^m + \epsilon)^{1/N}$ and $\|B_t\| \leq (\|B\|_k^m + \epsilon)^{1/M}$ for each $s,t$. Then
    \[ \|A \oplus B\|_{n+k}^m \leq \| (A_1 \oplus B_1) \|_{n+k} \| (A_2 \oplus B_2) \|_{n+k} \dots \|(A_N \oplus B_N) \|_{n+k} \leq \|A\|_n^m + \epsilon \]
    and hence we conclude $\|A \oplus B\|_{n+k}^m \leq \|A\|_n^m$. Moreover, since
    \[ \|A\|_n^m = \| \begin{pmatrix} I_n & 0 \end{pmatrix} \begin{pmatrix} A & 0 \\ 0 & B \end{pmatrix} \begin{pmatrix} I_n & 0 \end{pmatrix}^T \|_{n}^m \leq \|A \oplus B\|_{n+m}^m \]
    we conclude that $\|A \oplus B\|_{n+k}^m = \|A\|_n^m = \max(\|A\|_n^m, \|B\|_k^m)$.

    Finally, suppose we have an injective completely contractive unital product map $i: \X \to \bB(\H)$. Let $A \in M_n(\X)$ with $A \neq 0$, and assume $A = A_1 \odot A_2 \odot \dots \odot A_N$ is a permissible factorization. Then $i^{(n)}(A) = i^{(n)}(A_1) \odot i^{(n)}(A_2) \odot \dots \odot i^{(n)}(A_N)$. Since on $\bB(\H)$ the natural product is operator composition, we have that $\|i^{(n)}(A)\| \leq \|i^{(n)}(A_1)\| \dots \|i^{(n)}(A_N)\| \leq \|A_1 \| \dots \|A_N\|$. But since $i^{(n)}(A) \neq 0$ (by injectivity of $i$), the product $\|A_1 \| \dots \|A_N\|$ is bounded below by the positive value $\|i^{(n)}(A)\|$ and thus $0 < \|i^{(n)}(A)\| \leq \|A\|_n^m$. We conclude that $i$ is completely contractive and that $\{\|\cdot\|_n^m\}$ is an $\mathrm{L}^\infty$-matrix norm in this case.
\end{proof}

\begin{theorem} \label{thm: factorization norm is unital}
    Let $\X$ be a operator space with unit $e\in \X$, $\D$ a domain and $m : \D \rightarrow \X$ a partial product. Suppose that $\{\|\cdot\|_n^m\}$ is an $\mathrm{L}^\infty$-matrix \emph{norm} on $\X$. Then $(\X,e, \{\|\cdot\|_n^m \})$ is a unital operator space and $m$ is a completely contractive partial product, i.e. there exists $\pi: \X \to \bB(\H)$ completely isometric such that $\pi(e) = I$ and $\pi(m(a,b)) = \pi(a)\pi(b)$ for all $(a,b) \in \D$.
\end{theorem} 

\begin{proof}
    It suffices to check that $m$ is completely contractive. Suppose that $C \in M_n(\X)$, and that $(A,B)$ is a valid pair of matrices over $\X$ with $C = A \odot_m B$. Suppose further that $A = A_1 \odot A_2 \odot \dots \odot A_N$ and that $B = B_1 \odot B_2 \odot \dots \odot B_M$ are permissible factorizations. Then $C = A_1 \odot \dots \odot A_N \odot B_1 \odot \dots \odot B_M$ is a permissible factorization by Lemma \ref{l:permis}. Hence
    \[ \|C\|_n^m \leq (\|A_1\| \dots \|A_N\|)(\|B_1\| \dots \|B_M\|). \]
    Taking an infimum on the right shows that $\|C\|_n^m \leq \|A\|_n^m \|B\|_n^m$. Therefore $m$ is completely contractive.
\end{proof}

\subsection{Unital operator spaces}

Suppose that $\V$ is an operator space, $e \in \V$ is a unit vector, and $m$ is the trivial product defined by $m(e,x)=m(x,e)=x$ for all $x \in \V$. If $m$ is completely contracitve, then $\V$ is already a unital operator space, by Theorem \ref{thm: Unit characterization}. If not, the previous theorem implies that if the factorization norm 
\[ \|A\|_{n,k}^e := \inf \{ \|a_1\| \|a_2\| \dots \|a_n\| : A = a_1 \odot_m \dots \odot_m a_n \} \]
is an $\mathrm{L}^{\infty}$-matrix norm, then $(\V,\{\|\cdot\|_n^e\})$ is a unital operator space. In fact, it is the ``largest'' \emph{unital} operator space structure on $\V$ subordinate to the given operator space structure, as explained in the next result.

\begin{theorem} \label{thm: unital factorization norm}
    Suppose that $\V$ is an operator space with matrix norms $\{\|\cdot\|_n\}$, $e \in \V$ is a unit vector, and $m$ is the trivial product defined by $m(e,x)=m(x,e)=x$ for all $x \in \V$. If 
    \[ \|A\|_n^e := \inf \{ \|a_1\| \|a_2\| \dots \|a_n\| : A = a_1 \odot_m \dots \odot_m a_n \} \]
    is an $\mathrm{L}^{\infty}$-matrix norm, then $(\V,\{\|\cdot\|_n^e\})$ is a unital operator space with unit $e$. Moreover, if $\{\|\cdot\|_n^*\}$ is another matirx norm on $\V$ making $(\V, \{\|\cdot\|_n^*\}, e)$ a unital operator space, and if $\|x\|_n^* \leq \|x\|_n$ for all $x \in M_n(\V)$, then $\|x\|_n^* \leq \|x\|^e_n$ for all $x \in M_n(\V)$.
\end{theorem}

\begin{proof}
    That $(\V,\{\|\cdot\|_n^e\})$ is a unital operator space with unit $e$ follows from Theorem \ref{thm: factorization norm is unital}. Now suppose we are given $\{\|\cdot\|_n^*\}$ as stated. Since $(V, \{\|\cdot\|_n^*\}, e)$ is unital, there exists a linear map $\pi: \V \to \bB(\H)$ with $\|\pi^{(n)}(x)\| = \|x\|_n^*$ for all $x \in M_n(\V)$ and $\pi(e)=I$. Suppose $X \in M_{n,k}(\V)$ and $X = X_1 \odot_m \dots \odot_m X_N$ is a permissible factorization with respect to the trivial product. Because $\pi$ is unital, $\pi^{(n,k)}(X) = \pi^{n,n_1}(X_1) \pi^{(n_1,n_2)}(X_2) \dots \pi^{(n_{N-1},k)}(X_n)$. Hence
    \begin{eqnarray}
        \|X\|_{n,k}^* & = & \|\pi^{(n,k)}(X)\| \nonumber \\
        & \leq & \|\pi^{n,n_1}(X_1)\| \cdots \|\pi^{n_{N-1},k)}(X_N)\| \nonumber \\
        & = & \|X_1\|_{n,n_1}^* \dots \|X_N\|_{n_{N-1},k}^* \nonumber \\
        & \leq & \|X_1\|_{n,n_1} \dots \|X_N\|_{n_{N-1},k}. \nonumber
    \end{eqnarray}
    Taking an infimum shows that $\|X\|_{n,k}^* \leq \|X\|_{n,k}^e$.
\end{proof}

Now suppose that $(\X,e)$ is a unital operator space and that $\K \subseteq \X$ is the kernel of a unital completely contractive map. Then we may form the operator space quotient $(\X / \K)_{osp}$ and unital operator space quotient $(\X / \K)_{usp}$. The identity map $i: (\X / \K)_{osp} \to (\X / \K)_{usp}$ is completely contractive, but not necessarily completely isometric. The above theorem allows us to express the matrix norms on $(\X / \K)_{usp}$ in terms of the matrix norms on $(\X / \K)_{osp}$. In the case when $\X$ is an operator system, $(\X / \K)_{usp}$ is isometricaly isomorphic to the operator system quotient. Thus, for operator systems, we get an expression for the operator system quotient in terms of the operator space quotient.

\begin{corollary} \label{cor: Quotient factorization norm}
    Let $(\X,e)$ be a unital operator space and $\K \subseteq \X$ the kernel of a unital complete contraction. Let $\{\|\cdot\|_n^{osp} \}$ and $\{\|\cdot\|_n^{usp}\}$ denote operator space quotient and unital operator space quotient norms, respectively, on $\X / \K$. Then $\|A + M_n(\K)\|_n^{usp} = \|A + M_n(\K)\|_n^{e + \K}$ for all $A \in M_n(\X)$, where $\|\cdot\|^{e+\K}$ is the matrix norm defined in Theorem \ref{thm: unital factorization norm} according to the unit vector $e + \K \in X/ \K$.
\end{corollary}

\begin{proof}
    Since $i: (\X/\K)_{osp} \to (\X/\K)_{usp}$ is completely contractive and injective, and since $(\X/\K)_{usp}$ is unital with unit $e+\K$, the matrix norm $\|\cdot\|_n^{e+\K}$ makes $\X/\K$ into a unital operator space. Hence, by the universal property of the unital operator space quotient, the identity map $j: (\X/\K)_{usp} \to (\X/\K, \{\|\cdot\|_n^{e+\K}\})$ is completely contractive. That $j^{-1}: (\X/\K, \{\|\cdot\|_n^{e+\K}\}) \to (\X/\K)_{usp}$ is completely contractive follows from Theorem \ref{thm: unital factorization norm}. We conclude that $(\X/\K, \{\|\cdot\|_n^{e+\K}\}) = (\X/\K)_{usp}$.
\end{proof}

Explicitly, the above theorem tells us that when $x \in M_n((\X /\K)_{usp})$,
\[ \|X\|_n^{usp} = \inf \{ \|X_1\| \dots \|X_N\| : X = X_1 \odot \dots \odot X_n \}\]
where the infimum is over all permissible factorizations in $\X / \K$ with the trivial product and the norms on the right hand side are computed in the operator space $(\X / \K)_{osp}$.

\subsection{Commuting tensor product}

Let $\S$ and $\T$ be operator systems with units $e_{\S}$ and $e_{\T}$ respectively, and let $\V = \S \otimes \T$ denote the algebraic tensor product. Define a domain $\D$ and partial product $m: \D \to \V$ as for the commuting tensor product (see Construction \ref{const: c tensor}).

Given a matrix $A \in M_{n,k}(\V)$, we define
\[ \|A\|^c := \inf \{\|a_1\| \|a_2\| \dots \|a_n\| : A = a_1 \odot_m \dots \odot_m a_n\} \]
where the infimum is taken over all permissible factorizations of $A$ as matrices over $\V$ with the constraint that the entries of each matrix lie in either $\S \otimes \bC e_{\T}$ or $\bC e_{\S} \otimes \T$.

\begin{lemma} \label{lem: commuting factorization norm}
    The sequence $\{\|\cdot\|_n^c\}$ defines an $\mathrm{L}^\infty$ matrix norm.
\end{lemma}

\begin{proof}
    The proof that $\|\cdot\|^c$ is an $\mathrm{L}^{\infty}$-seminorm is similar to the proof of Proposition \ref{prop: factorization seminorm}, except that to show the $\mathrm{L}^{\infty}$-condition, we must show that if $A = A_1 \odot \dots A_N$ and $B = B_1 \odot_m \dots \odot_m B_N$, we may write $A \oplus B = (A_1 \oplus B_1) \odot_m \dots \odot_m (A_n \oplus B_n)$ with the entries of $(A_i \oplus B_i)$ all in either $S$ or $T$. This can be achieved by inserting copies of the identity matrix and rescaling so that all matrices in the factorizations have the same norm. To see that $\|\cdot\|^c$ is a norm, consider $\varphi: \S \otimes \T \to \S \otimes_c \T$. This map is clearly unital. To see that it is contractive, let $\pi: \S \otimes_c \T \to \bB(\H)$ be a unital complete isometry with $\pi(\S)$ and $\pi(\T)$ commuting in $\bB(\H)$. Then if $A = A_1 \odot_m \dots \odot_m A_n$ is permissible, it follows that
    \[ \|A\|_{\S \otimes_c \T} = \|\pi(A)\| = \|\pi(a_1) \pi(a_2) \dots \pi(a_n) \| \leq \|\pi(a_1)\| \dots \|\pi(a_n)\| = \|a_1 \| \dots \|a_n\|. \]
    It follows that $\|A\|_{\S \otimes_c \T} \leq \|A\|^c$. Hence $\varphi$ is completely contractive and injective (since $\S \otimes_c \T$ is algebraically isomorphic to the algebraic tensor product $\S \otimes \T$).
\end{proof}

\begin{theorem}
    Suppose that $\S$ and $\T$ are operator systems. Then $m$ is a completely contractive product on the unital operator system $(\S \otimes \T, e, \|\cdot\|^c)$.
\end{theorem}

\begin{proof}
    If $(A,B)$ is a valid pair, then $\|A \odot_m B \|^c \leq \|A\|^c \|B\|^c$ (as in the proof of Theorem \ref{thm: factorization norm is unital}). It follows that there exists a unital complete isometry $\pi: \S \otimes \T \to \bB(\H)$ such that $\pi(m(a,b)) = \pi(a)\pi(b)$. For $a \in \S$ and $b \in \T$, $\pi(a)\pi(b) = \pi(m(a,b)) = \pi(m(b,a)) = \pi(b) \pi(a)$. Since the restrictions of $\pi$ to $\S$ and $\T$ are unital and completely isometric, they are both self-adjoint. So \[ \pi(a \otimes b)^* = (\pi(a)\pi(b))^* = \pi(b)^* \pi(a)^* = \pi(b^*) \pi(a^*) = \pi(m(b^*,a^*)) = \pi(m(a^*,b^*)) = \pi(a^* \otimes b^*). \] So the image of $\S \otimes \T$ is an operator system with adjoint satisfying $(a \otimes b)^* = a^* \otimes b^*$.
\end{proof}

\begin{corollary} \label{cor: commuting tensor norm}
    The operator system $(\S \otimes \T, \|\cdot\|^c)$ is unitally isometrically isomorphic to $\S \otimes_c \T$.
\end{corollary}

\begin{proof}
    In the previous proof, we saw that $\pi(a)$ and $\pi(b)$ commute in any product representation of $\S \otimes \T$ on $\bB(\H)$. By the universal property of $\S \otimes_c \T$, the identity map from $\S \otimes_c \T$ to $(\S \otimes \T, \|\cdot\|^c)$ is completely contractive. However, we showed that the identity from $(\S \otimes \T, \|\cdot\|^c)$ to $\S \otimes_c \T$ is completely contractive in the Lemma above. So these operator systems are unitally isometrically isomorphic.
\end{proof}

It is an open problem whether or not the commuting tensor product is associative, i.e. if $(\S \otimes_c \T) \otimes_c \R \cong \S \otimes_c (\T \otimes_c \R)$ for all operator systems $\S, \T$, and $\R$. The matrix norms described above for the commuting tensor product suggest an analytic approach to this problem. However, it still remains unclear how to compare the matrix norms for the spaces $(\S \otimes_c \T) \otimes_c \R$ and $\S \otimes_c (\T \otimes_c \R)$. For example, an element $x \in \S \otimes_c (\T \otimes_c \R)$ admits a factorization $x = x_1 \odot x_2 \odot \dots \odot x_n$ with each $x_i$ an element of either $\S$ or $\T \otimes \R$ and such that $\|x\|$ is approximated by $\prod \|x_i\|$. Moreover, each $x_i$ lying in $\T \otimes \R$ admits a factorization $x_i = y_{i1} \odot \dots \odot y_{in_i}$ with each $y_{ij}$ an element of either $\T$ or $\R$ and such that $\|x_i\|$ is approximated by $\prod_j \|y_{ij}\|$. On the other hand, if one could show that every permissible factorization $x = z_1 \odot \dots \odot z_k$ of $x$ where each $z_i$ lies in $\S$, $\T$, or $\R$ and the multiplication is ``triple-commuting'' (so that $m(s,t)=m(t,s), m(s,r)=m(r,s)$, and $m(t,r)=m(r,t)$) can be refactored as a permissible factorization with respect to the product in $(\S \otimes_c \T) \otimes_c \R$, then we could compare the two norms. However, it is not clear if such triple-commuting factorizations can be regrouped to obtain a permissible factorization over $(\S \otimes_c \T) \otimes_c \R$. 

\subsection{Group operator systems}

First recall that a finitely presented group has a presentation with generators $\{a_1, \dots, a_n\}$ and relations of the form $a_i^x a_j^y = a_k^z$ where $x,y,z \in \{-1,0,1\}$. For instance, if $G$ is finitely presented and has a relation of the form $a_1 a_2 a_3 = a_4$, then by introducing an additional generator $u = a_1 a_2$ and new relations $a_1 a_2 = u$ and $u a_3 = a_4$, the relation $a_1 a_2 a_3 = a_4$ can be removed from the presentation and replaced with the relations $a_1 a_2 = u$ and $u a_3 = a_4$. By induction, similar modifications can be made on any finitely presented group to reduce the size of the words in the given relations by introducing new generators and relations.

Thus, suppose that $G$ is a finitely presented group with generators $\{a_1, \dots, a_n\}$ as in the form described in the previous paragraph. Let $\S = \text{span} \{I, a_1, \dots, a_n, a_1^{-1}, \dots, a_n^{-1}\}$ in the group algebra $\bC[G]$. Then $\S$ is a $*$-vector space with the usual involution given by \[ (\lambda_0 I + \sum \lambda_i a_i + \sum \mu_i a_i^{-1})^* = \overline{\lambda_0} I + \sum \overline{\lambda_i} a_i^{-1} + \sum \overline{\mu_i} a_i. \] Let $\D \subseteq \S \times \S$ denote all pairs $(s,t)$ such that $s,t \in \S$ and $st \in \S$, where the product is taken in the group algebra. Define $m: \D \to \S$ by $m(x,y) = xy$. Then $m$ is a partial product.

We wish to endow $\S$ with an operator space structure making $m$ into a completely contractive product. Such operator space norms exist --- for instance, $S$ inherits such an operator space norm by its inclusion into the universal group C*-algebra $C^*(G)$. However, by using factorization norms, we can describe such an operator space structure somewhat explicitly. In fact, we will show that the operator space norm we describe agrees with the one inherited by $\S$ in $C^*(G)$.

To define the desired operator space structure on $\S$, we will first define an intermediate operator space structure on $\S$. We will then replace the intermediate operator space structure with another by considering permissible factorizations and obtain the desired results by appealing to Proposition \ref{prop: factorization seminorm}. The intermediate operator space structure is defined as follows. First, endow the $(2n+1)$-dimensional Banach space $\ell^1_{2n+1}$ with its maximal operator space structure $\text{MAX}(\ell^1_{2n+1})$, i.e. by $M_n(\ell^1_{2n+1})$ with the matrix norms \[ \|x\|_n = \inf \{ \|\alpha\| \|\beta\| : x = \alpha \text{diag}(y_1, \dots, y_k) \beta \} \] where the infimum is over all $y_1, \dots, y_k \in \ell^1_{2n+1}$ with $\|y_i\| \leq 1$ for each $i=1,2,\dots,k$. By \cite{FKPTgroups2014} (c.f. \cite[Chapter 14]{paulsen2002completely}), this operator space can be identified with the span of $u_1, \dots, u_{2n+1}$ in the universal group C*-algebra for the free group on $2n+1$ generators. Let $\{e_{-n}, \dots, e_{-1}, e_0, e_1, \dots e_n\}$ denote the canonical basis for $\ell^1_{2n+1}$. To obtain an operator space structure on $\S$, define $\pi: \ell^1_{2n+1} \to \S$ by letting $\pi(e_0) = I$ and, for $i > 0$, letting $\pi(e_i) = a_i$ and $\pi(e_{-1}) = a_i^{-1}$ then extend linearly. Let $\J = \ker(\pi)$ (which may be trivial). Then we obtain an intermediate operator space structure on $\S$ by identifying $\S$ with the operator space quotient $\text{MAX}(\ell^1_{2n+1})/\J$.

\begin{lemma} \label{lem: intermediate norm}
    Let $G$ and $\S$ be as described above. Then the linear map from $\hat{\pi}: \text{MAX}(\ell^1_{2n+1})/\J \to C^*(G)$ satisfying $\hat{\pi}(e_0 + \J) = I$ and, for $i > 0$, $\hat{\pi}(e_i + \J) = a_i$ and $\hat{\pi}(e_{-i} + \J) = a_i^{-1}$, is completely contractive.
\end{lemma}

\begin{proof}
    Identifying $e_{-n}, \dots, e_{-1}, e_0, e_1, \dots, e_n$ with the generators of $\mathbb{F}_{2n+1}$, the universal property of the free group implies that the map $\pi: e_0 \mapsto I, e_i \mapsto a_i, e_{-i} \mapsto a_i^{-1}$ extends to a group homomorphism from $\mathbb{F}_n$ to $G$. Hence, $\pi$ further extends to a $*$-homomorphism $\pi: C^*(\mathbb{F}_{2n+1}) \to C^*(G)$, which is completely contractive. Identifying $\text{MAX}(\ell^1_{2n+1})$ with the linear span of $e_{-n}, \dots, e_n$, we see that the restriction of $\pi$ to $\text{MAX}(\ell^1_{2n+1})$ is completely contractive. As $\J$ is the kernel of this restriction, $\pi'$ is completely contractive by the universal property of the operator space quotient.
\end{proof}

Letting $\{\|\cdot\|_n\}$ denote the matrix norms on $\S$ given by the identification with $\text{MAX}(\ell^1_{2n+1})/\J$ described above, we now consider the induced factorization semi-norms
\[ \|A\|_n^m := \inf \{ \|a_1\| \|a_2\| \dots \|a_n\| : A = a_1 \odot_m \dots \odot_m a_n \} \]
where the infimum is over all permissible factorizations.

\begin{theorem} \label{t:factorization-norm-group}
    The factorization seminorm $\{\|\cdot\|_n^m\}$ defined above is an $\mathrm{L}^{\infty}$-matrix norm. Moreover, this matrix norm agrees with the matrix norm $\S$ inherits from the universal group C*-algebra $C^*(G)$.
\end{theorem}

\begin{proof}
    From Proposition \ref{prop: factorization seminorm}, $\{\|\cdot\|_n^m\}$ is a matrix seminorm, since $\{\|\cdot\|\}$ is an operator space norm on $\S$. Notice that the embedding $i: \S \to C^*(G)$ is a product map. By Lemma \ref{lem: intermediate norm}, $i$ is completely contractive when $\S$ is endowed with the intermediate matrix norm $\{\|\cdot\|\}$. It follows that the factorization seminorms $\{\|\cdot\|_n^m\}$ is an $\mathrm{L}^{\infty}$-matrix \emph{norm} and that $i$ is completely contractive when $\S$ is endowed with the matrix norm $\{\|\cdot\|_n^m\}$, by Proposition \ref{prop: factorization seminorm}.

    Finally, we show that the matrix norm $\{\|\cdot\|_n^m\}$ agrees with the operator space structure $\S$ inherits from the embedding $i: \S \to C^*(G)$. We have already shown that $i$ is completely contractive when $\S$ is endowed with $\{\|\cdot\|_n^m\}$, so we wish to show that $i^{-1}$ is also completely contractive on $i(\S)$. By Theorem \ref{thm: factorization norm is unital}, we see that $(\S,\{\|\cdot\|_n^m\})$ is a unital operator space and $m$ is a completely contractive product. By Corollary \ref{cor: abstract product characterization}, there exists a complete isometry $\rho: \S \to \bB(\H)$ such that $\rho(I) = I_{\H}$, $\rho(a_i)$ is unitary for each $i$ (since $m(a_i,a_i^{-1})=m(a_i^{-1},a_i)=I$), and $\rho(a_i^x)\rho(a_j^y) = \rho(a_k^z)$ for each relation $a_i^x a_j^y = a_k^z$ (since $m(a_i^x, a_j^y)=a_k^z)$). Therefore the group generated by $\{\rho(a_1), \dots, \rho(a_n)\}$ in $\bB(\H)$ is isomorphic to $G$. By the universal property of $C^*(G)$, the map $j: a_i \mapsto \rho(a_i)$ extends to a $*$-homomorphism $\psi: C^*(G) \to \bB(\H)$. But the restriction of this $*$-homomorphism to $i(\S)$ is just $i^{-1}$ (upon identifying $S$ completely isometrically with $\rho(\S) \subseteq \bB(\H)$). It follows that $i^{-1}$ is completely contractive.
\end{proof}

\subsection{Operator systems spanned by projections}

In \cite{AraizaRussellProjs}, an abstract characterization was presented for operator systems spanned by their unit and a finite number of projections $p_1, \dots, p_n$ satisfying given linear relations. These characterizations were given in terms of the matrix ordering of the operator system, which is constructed as an inductive limit of matrix orderings. Here, we give an alternative construction in terms of the matrix norms using permissible factorizations. As a corollary, we obtain new matrix norm characterizations for quantum commuting correlations.

For the reader's convenience, we summarize the construction described in \cite{AraizaRussellProjs}. Suppose there exists a Hilbert space $\H$ and projections $P_1, \dots, P_n \in \bB(\H)$ which satisfy linear relations of the form $\sum_{i=0}^n \alpha_{i,j} P_i = 0$ for $j=1,\dots,m$, where we let $P_0 := I$ for convenience. Let $\V$ denote the universal $*$-vector space spanned by self-adjoint vectors $p_0, p_1, \dots, p_n$ satisfying the same relations $\sum_{i=0}^n \alpha_{i,j} p_i = 0$. If we let $C$ denote the cone generated by $\{p_0, p_1, \dots, p_n\}$, then it follows that $C$ is a proper cone with order unit $e := p_0$. Indeed, the cone $C$ is a subcone of the proper cone of positive operators in the concrete operator system spanned by $I, P_1, \dots, P_n$ in $\bB(\H)$ and the map $\pi: p_i \mapsto P_i$ for all $i=0,1,\dots,n$ is injective.

Let $\{\|\cdot\|_n\}_{n=1}^{\infty}$ denote the matrix norm on the operator system $\text{OMAX}(\V)$ (see \cite{PaulsenTodorovTomforde2011OSS} for this definition). Let $\V_i = \text{span}\{I,p_i\} \subseteq \V$ for each $i=1,2,\dots,n$. Setting $\D = (\bC \times \V) \cup (\V \times \bC) \cup_{i=1}^n (\V_i \times \V_i)$, we extend the trivial product by defining the partial product $m: \D \to \V$ to satisfy $m(\lambda p_i, \mu p_i) = \lambda \mu p_i$ and extending to a unital bilinear map on $\D$. Since $\pi: \V \to \bB(\H)$ is ucp, and since each $p_i$ is a projection, $\pi$ is a product map. It follows from Proposition \ref{prop: factorization seminorm} that $\pi$ is completely contractive with respect to the factorization seminorm
\[ \|A\|_{n,k}^m := \inf \{ \|a_1\| \|a_2\| \dots \|a_n\| : A = a_1 \odot_m \dots \odot_m a_n \} \]
where the infimum is over all permissible factorizations. Since the Hilbert space $\H$ and projections $P_1, \dots, P_n$ were arbitrary in the above arguments, we have proven the following.

\begin{theorem} \label{thm: projection op sys}
    Let $\V$ denote the operator system constructed above. Then $(\V, \{\|\cdot\|_n^m\})$ is an operator system with completely contractive partial product $m$. Moreover, if $\H$ is a Hilbert space and $P_1, \dots, P_n \in \bB(\H)$ are projections satisfying the relations $\sum_{i=1}^n \alpha_{i,j} P_i = 0$ for each $j=1,\dots,m$, then the map $\pi: \V \to \bB(\H)$ given by $\pi(p_i)=P_i$ and $\pi(e)=I$ is a ucp product map.
\end{theorem}

As a corollary to the above result, we obtain a new operator space characterization for the set of quantum commuting correlations in terms of factorization norms. We summarize some definitions and known results concerning quantum commuting correlations before stating this result.

Let $n, k \in \mathbb{N}$ be fixed positive integers. A (bipartite) \textit{quantum commuting correlation} is a tuple $( \ p(a,b|x,y) : a,b \in [k], x,y \in [n] \ )$ (where $[m] := \{1,2,\dots,m\}$ for an integer $m \in \mathbb{N}$) arising in the following manner: there exists a C*-algebra $\A$, projection valued measures $\{E_{x,a}\}_{a=1}^k, \{F_{y,b}\}_{b=1}^k$ in $\A$ for each $x,y \in [n]$ satisfying $[E_{x,a},F_{y,b}] = 0$ for all $a,b \in [k]$, and a state $\varphi: \A \to \mathbb{C}$ such that $p(a,b|x,y) = \varphi(E_{x,a} F_{y,b})$ for all $a,b \in [k]$ and $x,y \in [n]$. Quantum commuting correlations arise as joint probability distributions modeling the simultaneous measurements of two independent observers sharing a possibly entangled quantum state in the Haag-Kastler model of quantum mechanics \cite{junge2011connes}. Distinguishing these correlations from those arising in the older tensor-product model of quantum mechanics was shown to be equivalent to Connes' embedding problem in \cite{ozawa2013connes}.

To obtain quantum commuting correlations from operator systems, we recall the concept of a \textit{quantum commuting operator system} from \cite{AraizaRussellAbs23}. This is an operator system $\S$ spanned by positive operators $Q(a,b|x,y)$ satisfying the following conditions:
\begin{enumerate}
    \item for each $x,y \in [n]$, $\sum_{a,b} Q(a,b|x,y) = I$,
    \item the marginal operator $E(a|x) := \sum_{b=1}^k Q(a,b|x,y)$ is well-defined, i.e. the sum is independent of the choice of $y \in [n]$.
    \item the marginal operator $F(b|y) := \sum_{a=1}^k Q(a,b|x,y)$ is well-defined, i.e. the sum is independent of the choice of $x \in [n]$.
    \item the image of $Q(a,b|x,y)$ is a projection in $C^*_e(\S)$ for each $a,b \in [k]$ and $x,y \in [n]$.
\end{enumerate}
It was shown in \cite{AraizaRussellAbs23} that $p(a,b|x,y)$ is a quantum commuting correlation if and only if there exists a quantum commuting operator system $\S$ and a state $\varphi: \S \to \mathbb{C}$ such that $p(a,b|x,y) = \varphi(Q(a,b|x,y))$ for all $a,b \in [k]$ and $x,y \in [n]$.

The condition (4) for a quantum commuting operator system was abstractly characterized in terms of the order structure of the operator system $\S$ in \cite{AraizaRussellAbs23}. It was also shown in \cite{ARTApublishedt21} how to construct a universal quantum commuting operator system $\S_{n,k}$ for which every quantum commuting correlation is realized via $p(a,b|x,y) = \varphi(Q(a,b|x,y))$ for some state $\varphi$ on $\S_{n,k}$.

We wish to give an alternative construction for $\S_{n,k}$ by describing its matrix norms as factorization norms. In fact, such a construction is a corollary of Theorem \ref{thm: projection op sys}. This is because $\S_{n,k}$ is precisely an operator system spanned by projections satisfying the linear ``non-signalling conditions'' (i.e. the conditions (1), (2), and (3) above) on the generators $Q(a,b|x,y)$. Thus, by extending the trivial partial product by specifying the partial product $m$ to satisfy $m(Q(a,b|x,y),Q(a,b|x,y)) = Q(a,b|x,y)$, we can equip the universal vector space $\V$ (described above Theorem \ref{thm: projection op sys}) with the matrix norm $\{\|\cdot\|_n^m\}$ making it into a product system surjecting onto every quantum commuting operator system. Since each operator $Q(a,b|x,y)$ is a projection in this operator system, we have the following.

\begin{corollary} \label{c:correlation-system-factorization}
    Given $n,k \in \mathbb{N}$, the $*$-vector space $\S_{n,k}$ satisfying conditions (1), (2), and (3) above, equipped with the matrix norm $\{\|\cdot\|_n^m\}$, defines an operator system satisfying the following universal property: a correlation $\{p(a,b|x,y)\}$ is quantum commuting if and only if there exists a state $\varphi: \S_{n,k} \to \bC$ such that $p(a,b|x,y) = \varphi(Q(a,b|x,y))$ for all $a,b \in [k]$ and $x,y \in [n]$.
\end{corollary}

\subsection{Tracial states}

In this section, we wish to characterize when a state $\tau: \V \to \mathbb{C}$ on a unital operator space with a given partial product extends to a tracial state on the $C^*_m(\V)$. We will see that this occurs whenever $\tau$ is contractive with respect to a certain factorization norm defined on the free algebra $\mathcal{F}_m(\V)$ containing $\V$.

Let $\V$ be a unital operator space and let $m: \D \to \V$ be a completely contractive partial product. Then $\V$ is non-degenerate, since the universal mapping $\pi: \mathcal{F}_m(\V) \to C^*_m(\V)$ maps $\V + \J_m$ to the image of $\V$ in $C^*_m(\V)$. Identifying $\F_m(\V)$ with its image in $C^*_m(\V)$, we see that $\F_m(\V)$ is dense in the unital operator algebra $\A_m(\V)$ generated by $\V$ in $C^*_m(\V)$. If $\V$ is an operator system, then $\F_m(\V)$ is dense in $C^*_m(\V)$ since $\A_m(\V) = C^*_m(\V)$. Since the operator space structure of $\V + \V^*$ is uniquely determined by the operator space structure of $\V$, we can always extend $m$ to a completely contractive partial product of $\V + \V^*$, so we may assume $\V$ is an operator system without loss of generality. To define a factorization norm on $\V$, we will need to first consider a factorization norm on all of $\mathcal{F}_m(\V)$.

\begin{proposition} \label{prop: norm fm(V)}
    Let $\V$ be an operator system with completely contractive partial product $m$. For any $x \in M_n(\F_m(\V))$, define
    \[ \|x\|_n = \inf \{\|x_1\| \dots \|x_N\| : x = x_1 \odot_m \dots \odot_m x_N \} \]
where the infimum is taken over all permissible factorizations with each $x_i$ a matrix over $\V$. Then $\{\|\cdot\|_n\}$ is an $\mathrm{L}^{\infty}$-matrix seminorm on $\F_m(\V)$ and the universal map $\pi: \mathcal{F}_m(\V) \to C^*_m(\V)$ is completely isometric, i.e. $\|\pi^{(n)}(x)\| = \|x\|_n$ for all $x \in M_n(\F_m(\V))$.
\end{proposition}

Note that we do not assert that each $\|\cdot\|_n$ is a norm since $\pi$ need not be injective even when its restriction to $\V$ is injective.

\begin{proof}
    Suppose that $x \in M_n(\F_m(\V))$ and $x = x_1 \odot_m \dots \odot_m x_N$ is a permissible factorization. Since $\pi$ is a homomorphism,
    \[ \pi^{(n)}(x) = \pi^{(n,n_2)}(x_1) \pi^{(n_2,n_3)}(x_2) \dots \pi^{(n_N,n)}(x_N). \] It follows that $\pi$ is completely contractive. On the other hand, repeating the arguments of Proposition \ref{prop: factorization seminorm}, we see that $\{\|\cdot\|_n\}$ is an $\mathrm{L}^{\infty}$-matrix seminorm. Let $\K = \{x \in \F_m(\V) : \|x\|_1 = 0\}$. Then $\F_m(\V)/\K$ is a operator space (and $\K$ is an ideal in $\F_m(\V)$). Repeating the arguments of Corollary \ref{cor: Quotient factorization norm}, we see that $m$ is completely contractive, so that $\F_m(\V)/\K$ generates a product C*-cover for $\V$. Since the universal mapping $\pi_m: C^*_m(\V) \to C^*(\F_m(\V)/\K)$ is completely contractive, we deduce that $\pi$ is completely isometric.
\end{proof}

Let $\V$ be an operator system. Let $\J_{\tau}$ denote the linear span all commutators in $\F_m(\V)$, i.e. $\J_{\tau}$ is the span of elements $xy - yx$ where $x$ and $y$ are words in elements of $\V$.

\begin{proposition} \label{prop: trace characterization}
    Let $\V$ be an operator system with completely contractive partial product $m$. For each $x \in \F_m(\V)$, let $\|x\|_{\tau} = \inf \{ \|x+y\| : y \in \J_{\tau}\}$. Then $\|\cdot\|_{\tau}$ is a seminorm on $\mathcal{F}_m(\V)$, and hence its restriction to $\V$ is a seminorm. Moreover, a unital linear map $\varphi: C^*_m(\V) \to \mathbb{C}$ is tracial if and only if its restriction to $\mathcal{F}_m(\V)$ is contractive with respect to $\|\cdot\|_{\tau}$.
\end{proposition}

\begin{proof}
    It suffices to show that $\J_{\tau}$ is dense in the closed linear span of all commutators in $C^*_m(\V)$, since a linear map $\varphi: C^*_m(\V) \to \mathbb{C}$ is tracial if and only if it is identically zero on the set of commutators in $C^*_m(\V)$. Since $\V$ is an operator system, $\F_m(\V)$ is a $*$-algebra. Since both $\F_m(\V)$ and $C^*_m(\V)$ are generated by $\V$, and since the map sending $\F_m(\V)$ into $C^*_m(\V)$ has dense range, we see that every element of $C^*_m(\V)$ can be approximated in norm by the image of an element of $\F_m(\V)$. In particular, every commutator $[x,y]$ in $C^*_m(\V)$ can be approximated in norm by a commutator $[x',y']$ with $x', y' \in \F_m(\V)$.
\end{proof}

Since $\|\cdot\|_{\tau}$ is defined entirely in terms of the matrix norms on $\V$ and the partial product $m$ by Proposition \ref{prop: norm fm(V)} and Proposition \ref{prop: trace characterization}, we have an intrinsic characterization of the set of states $\varphi: \V \to \bC$ which extend to tracial states on the product C*-cover $C^*_m(\V)$.

\begin{corollary} \label{c:tracial-extension}
    A unital linear map $\varphi: \V \to \bC$ extends to a tracial state $\widetilde{\varphi}: C^*_m(\V) \to \bC$ if and only if $\varphi$ is contactive with respect to the seminorm $\|\cdot\|_{\tau}$ on $\V$.
\end{corollary}

Note that the seminorm $\|\cdot\|_{\tau}$ may be trivial; for instance, it may satisfy $\|e\|_{\tau} = 0$ where $e$ is the unit of $\V$. In this case, $C^*_m(\V)$ lacks tracial states. For instance, let $\V_n$ be the Cuntz system considered in Example \ref{ex: Cuntz system}. Then $\V_n$ has no tracial states. This is because $nI = \sum s_i^* s_i$, whereas $I = \sum s_i s_i^*$. Hence any state $\varphi: \V_n \to \bC$ satisfying $\varphi(m(s_i,s_i^*)) = \varphi(m(s_i^*,s_i))$ would also satisfy $\varphi(I)=0$. Thus, $C^*_m(\V_n) = \O_n$ has no tracial states. So the obstruction to tracial states is evident at the level of the operator system $\V_n$.

\begin{remark}
    \emph{Corollary \ref{c:tracial-extension} shows that there is a one-to-one correspondance between states on $\V$ which extend to traces on $C^*_m(\V)$ and states that are contractive with respect to the norm $\|\cdot\|_{\tau}$. This provides an ``intrinsic'' meaning to tracial states on the operator system $\V$.} 
\end{remark}

\subsection{Corners of synchronous correlations}

We conclude with an application of the above characterization for tracial states to quantum correlation sets. A correlation $p(a,b|x,y)$ is called \textit{synchronous} if $p(a,b|x,x)=0$ whenever $a \neq b$. We let $C_{qc}^s(n,k)$ denote the set of synchronous quantum commuting correlations. This is easily seen to form a closed convex subset of $C_{qc}(n,k)$. Likewise, the synchronous quantum correlations $C_q^s(n,k)$ form a relatively closed convex subset of $C_q(n,k)$. It was shown in \cite{KimPaulsenSchafhauser18} that the equality $\overline{C_q^s(n,k)} = C_{qc}^s(n,k)$ for all $n$ and $k$ is equivalent to Tsirelson's conjecture (and hence to an affirmative answer to Connes' Embedding Problem). This fact is utilized in \cite{ji2020mip} in their refutation Tsirelson's conjecture. Whether or not $\overline{C_q^s(n,k)} = C_{qc}^s(n,k)$ is generally unknown for small values of $n$ and $k$, although we do have the equality $C_q^s(n,k) = C_{qc}^s(n,k)$ when $k=2$ and $n \leq 3$ \cite{RussellGeometry20}.

Tracial states play a crucial role in the study of synchronous quantum commuting correlations, due to the following chracterization of Paulsen, Severini, Stahlke, Todorov and Winter \cite{paulsen2016estimating}.

\begin{theorem}[Theorem 5.5 and Corollary 5.6 of \cite{paulsen2016estimating}] \label{thm: Paulsen-Severini et al}
    Let $\H$ be a Hilbert space with projection-valued measures $\{E_{x,a}\}_{a \ in [k]}, \{F_{y,b}\}_{b \in [k]} \subseteq \mathbb{B}(\H)$ satisfying $[E_{x,a},F_{y,b}] = 0$ for all $x, y \in [n]$ and $a, b \in [k]$, and a unit vector $h \in \H$, and define \[ p(a,b|x,y) = \langle E_{x,a} F_{y,b} h, h \rangle \] for all $x,y \in [n]$ and $a,b \in [k]$. Let $\mathcal{A} = C^*(\{E_{x,a}\})$ and let $\mathcal{B} = C^*(\{F_{y,b}\})$. If $p \in C^s_{qc}(n,k)$, then the state defined by $\varphi(T) = \langle T h,h \rangle$ is a tracial state when restricted to both $\mathcal{A}$ and $\mathcal{B}$, and hence is a tracial state on the product C*-algebra $\mathcal{A} \mathcal{B} \subseteq \mathbb{B}(\H)$. Moreover, 
    \[ p(a,b|x,y) = \varphi(E_{x,a}F_{y,b}) = \varphi(E_{x,a}E_{y,b}) = \varphi(F_{x,a} F_{y,b}) \]
    for all $x,y \in [n]$ and $a, b \in [k]$. Conversely, if $\mathcal{A}$ is a C*-algebra, $\{E_{x,a}\}_{a \in [k]} \subseteq \mathbb{B}(\H)$ is a projection-valued measure for each $x \in [n]$, and $\varphi: \mathcal{A} \to \mathbb{C}$ is a tracial state, then $p(a,b|x,y) := \varphi(E_{x,a} F_{y,b})$ is a synchronous quantum commuting correlation.
\end{theorem}

Suppose that we have C*-algebras $\mathcal{A}$ and $\mathcal{B}$, projection-valued measures $\{E_{x,a}\} \subseteq \mathcal{A}, \{F_{y,b}\} \subseteq \mathcal{B}$, and let $\varphi: \mathcal{A} \otimes_{max} \mathcal{B} \to \mathbb{C}$ be a tracial state. Applying the GNS construction to $\varphi$, we obtain the situation described in Theorem \ref{thm: Paulsen-Severini et al}; i.e. we obtain a Hilbert space $\H$, projection-valued measures $\{\pi(E_{x,a})\}, \{\pi(F_{y,b})\} \subseteq \mathbb{B}(\H)$ satisfying $[\pi(E_{x,a}),\pi(F_{y,b})] = 0$ for all $x, y, a$, and $b$, and a unit vector $h \in \H$ such that \[ p(a,b|x,y) = \varphi(E_{x,a} F_{y,b}) = \langle \pi(E_{x,a}) \pi(F_{y,b}) h, h \rangle \]
for all $x, y, a,$ and $b$. Although the corresponding correlation $p \in C_{qc}(n,k)$ need not be synchronous, it can be realized as the corner of a larger synchronous correlation in the following way. For each $x = 1,2, \dots, n$, define $R_{x,a} = E_{x,a}$, and for each $x = n+1, n+2, \dots, 2n$, define $R_{x,a} = F_{x,a}$. Then we obtain a correlation $q \in C^s_{qc}(2n,k)$ by setting $q(a,b|x,y) = \varphi(R_{x,a} R_{y,b})$, since $\varphi$ is tracial on $\mathcal{A} \otimes_{max} \mathcal{B}$ and by applying Theorem \ref{thm: Paulsen-Severini et al}. Now we can realize the original correlation $p$ as the ``corner'' of $q$ in the sense that $p(a,b|x,y) = q(a,b|x,y+n)$ for all $a,b \in [k]$ and $x,y \in [n]$.

Formally, we say that a quantum commuting correlation $\{p(a,b|x,y)\} \in C_{qc}^s(n,k)$ is a \textit{synchronous corner} if $p(a,b|x,y)=q(a,b,x,y+n)$ for some $\{q(a,b|x,y)\} \in C_{qc}^s(2n,k)$. We write $C_{qc}^{sc}(n,k)$ for the quantum commuting synchronous corners. We also define $C_q^{sc}(n,k) = C_{qc}^{sc}(n,k) \cap C_{q}(n,k)$, $C_{qa}^{sc} := C_{qc}^{sc}(n,k) \cap C_{qa}(n,k)$, and $C_{loc}^{sc}(n,k) = C_{qc}^{sc}(n,k) \cap C_{loc}(n,k)$. It is evident that $C_{r}^{sc}(n,k)$ is a relatively closed convex subset of $C_{r}(n,k)$ for each $r \in  \{loc, q, qa, qc\}$. The question of whether or not $C_r^{sc}(n,k) = C_r(n,k)$ was posed in \cite[Question 4.2]{maxEntangle} for each $r \in  \{loc, q, qa, qc\}$. It was proven that equality holds for $r=loc$ (and also for \textit{non-signalling} correlations) in \cite[Theorem 4.4]{maxEntangle}. It was also shown that every correlation in $C_q^{sc}(n,k)$ can be approximated by a correlation implemented with a maximally entangled state \cite[Theorem 4.3]{maxEntangle}. Finally, it was shown in \cite{maxEntangle}, as a consequence of \cite{KimPaulsenSchafhauser18} and \cite{ColadangeloStark20}, that the synchronous corners of the so-called \textit{quantum spatial} correlations do not agree with the set of all quantum spatial correlations.

By the above discussions, together with Theorem \ref{thm: Paulsen-Severini et al}, we have the following characterization of the set $C_{qc}^{sc}(n,k)$.

\begin{theorem}
    Let $p = \{p(a,b|x,y)\}$ be a correlation. Then the following statements are equivalent.
    \begin{enumerate}
        \item $p \in C_{qc}^{sc}(n,k)$.
        \item There exist C*-algebras $\mathcal{A}$ and $\mathcal{B}$, projection-valued operators $\{E_{x,a}\} \subseteq \mathcal{A}$ and $\{F_{y,b}\} \subseteq \mathcal{B}$ and a tracial state $\tau: \mathcal{A} \otimes_{max} \mathcal{B} \to \mathbb{C}$ such that $p(a,b|x,y) = \tau(E_{x,a} \otimes F_{y,b})$ for all $x,y \in [n]$ and $a,b \in [k]$.
        \item There exists a tracial state $\tau: \S_{n,k} \to \mathbb{C}$ such that $p(a,b|x,y) = \tau(Q(a,b|x,y))$, where $\S_{n,k}$ is the operator system considered in Corollary \ref{c:correlation-system-factorization}.
    \end{enumerate}
\end{theorem}

\begin{proof}
    The equivalence of (1) and (2) follows from Theorem \ref{thm: Paulsen-Severini et al} and the subsequent discussions above. We will show the equivalence of (2) and (3).

    Suppose that (2) holds, and let $\{Q(a,b|x,y)\}$ denote the generators of $\S_{n,k}$. By Theorem \ref{thm: projection op sys}, there exists a ucp product map $\pi: \S_{n,k} \to \mathcal{A} \otimes_{max} \mathcal{B}$ satisfying $\pi(Q(a,b|x,y)) = E_{x,a} \otimes F_{y,b}$. Consider the map $\tau' = \tau \circ \pi: \S_{n,k} \to \mathbb{C}$. Since $\pi$ is a product map, $\pi$ extends to a $*$-homomorphism on $C^*_m(\S_{n,k})$ by Theorem \ref{thm: c-star-m universal property}. Since $\tau$ is a tracial state, so is $\tau'$. So $\tau'$ is a tracial state on $\S_{n,k}$ by Corollary \ref{c:tracial-extension}.

    Finally, suppose that (3) holds. By Corollary \ref{c:tracial-extension} again, $\tau$ extends to a tracial state on $C^*_m(\S_{n,k})$. Then each $Q(a,b|x,y)$ is a projection in $C^*_m(\S_{n,k})$. As shown in \cite{AraizaRussellAbs23}, each $Q(a,b|x,y)$ decomposes as $Q(a,b|x,y)=E(a|x)F(y|b)$, where $\{E(a|x)\}_a, \{F(y|b)\}_b$ are mutually commuting projection-valued measures. Let $\mathcal{A} = C^*(\{E(a|x)\})$ and $\mathcal{B} = C^*(\{F(b|y)\})$. Then $\mathcal{A}$ and $\mathcal{B}$ are mutually commuting C*-subalgebras of $C^*_m(\S_{n,k})$. By the universal property of the maximal C*-algebra tensor product, there is a $*$-homomorphism $\pi: \mathcal{A} \otimes_{max} \mathcal{B} \to C^*_m(\S_{n,k})$ satisfying $\pi(E(a|x)) \otimes F(b|y)) = E(a|x)F(b|y)=Q(a,b|x,y)$. Hence $\tau' := \pi \circ \tau: \mathcal{A} \otimes_{max} \mathcal{B} \to \mathbb{C}$ is a tracial state as desired.
\end{proof}

\bibliographystyle{plain}
\bibliography{Refs}

\begin{thebibliography}{10}

\bibitem{Agler88}
J.~Agler.
\newblock An abstract approach to model theory.
\newblock {\em Surveys of some recent results in operator theory}, II:1--23,
  1998.

\bibitem{maxEntangle}
E.~Alhajjar and T.~Russell.
\newblock Maximally entangled correlation sets.
\newblock {\em Houston Journal of Mathematics}, 46(2):357--376, 2020.

\bibitem{ARTApublishedt21}
R.~Araiza, T.~Russell, and M.~Tomforde.
\newblock A universal representation for quantum commuting correlations.
\newblock {\em Annales Henri Poincare}, 23:4357--4397, 2022.

\bibitem{AraizaRussellAbs23}
Roy Araiza and Travis Russell.
\newblock An abstract characterization for projections in operator systems.
\newblock {\em Journal of Operator Theory}, 90(1):41--72, 2023.

\bibitem{AraizaRussellProjs}
Roy Araiza and Travis Russell.
\newblock Operator systems generated by projections.
\newblock {\em International Mathematics Research Notices}, 2025:rnae283, 2025.

\bibitem{Arveson08}
William Arveson.
\newblock The noncommutative choquet boundary.
\newblock {\em Journal of the American Mathematical Society}, 21(4):1065--1084,
  2008.

\bibitem{Arveson69}
William~B. Arveson.
\newblock Subalgebras of {$C^{\ast} $}-algebras.
\newblock {\em Acta Math.}, 123:141--224, 1969.

\bibitem{BlecherMultipliers}
David Blecher.
\newblock Multipliers and dual operator algebras.
\newblock {\em Journal of Functional Analysis}, 183(4):498--525, 2001.

\bibitem{BlecherOperatorAlgebras}
David~P Blecher and Christian Le~Merdy.
\newblock {\em Operator algebras and their modules}.
\newblock Oxford University Press, 2004.

\bibitem{BlecherMagajna10}
David~P. Blecher and Bojan Magajna.
\newblock Dual operator systems.
\newblock {\em Bulletin of the London Mathematical Society}, 43(2):311--320,
  2011.

\bibitem{blecher2011metric}
David~P Blecher and Matthew Neal.
\newblock Metric characterizations of isometries and of unital operator spaces
  and systems.
\newblock {\em Proceedings of the American Mathematical Society},
  139(3):985--998, 2011.

\bibitem{BOCA1991251}
Florin Boca.
\newblock Free products of completely positive maps and spectral sets.
\newblock {\em Journal of Functional Analysis}, 97(2):251--263, 1991.

\bibitem{CHOIEffros1977}
Man-Duen Choi and Edward~G. Effros.
\newblock Injectivity and operator spaces.
\newblock {\em Journal of Functional Analysis}, 24(2):156--209, 1977.

\bibitem{CESCBMultilinear}
E.~Christensen, E.G. Effros, and A.M. Sinclar.
\newblock Completely bounded multilinar maps and {C}*-algebraic cohomology.
\newblock {\em Invent. Math.}, 90:279--296, 1987.

\bibitem{ColadangeloStark20}
A.~Coladangelo and J.~Stark.
\newblock An inherently infinite-dimensional quantum correlation.
\newblock {\em Nat Commun}, 11(3335), 2020.

\bibitem{DavidsonCstarBook}
Kenneth~R. Davidson.
\newblock {\em {$C^*$}-algebras by example}, volume~6 of {\em Fields Institute
  Monographs}.
\newblock American Mathematical Society, Providence, RI, 1996.

\bibitem{DritschelMccullough05}
Michael~A. Dritschel and Scott~A. McCullough.
\newblock Boundary representations for families of representations of operator
  algebras and spaces.
\newblock {\em Journal of Operator Theory}, 53(1):159--167, 2005.

\bibitem{FKPTgroups2014}
Doug Farenick, Ali~S. Kavruk, Vern~I. Paulsen, and Mark Tomforde.
\newblock Operator systems from discrete groups.
\newblock {\em Commun. Math. Phys.}, 329:207–238, 2014.

\bibitem{hamana1979injective}
Masamichi Hamana.
\newblock Injective envelopes of operator systems.
\newblock {\em Publications of the Research Institute for Mathematical
  Sciences}, 15(3):773--785, 1979.

\bibitem{ji2020mip}
Zhengfeng Ji, Anand Natarajan, Thomas Vidick, John Wright, and Henry Yuen.
\newblock {MIP}{*}= {RE}.
\newblock {\em arXiv preprint:2001.04383}, 2020.

\bibitem{junge2011connes}
Marius Junge, Miguel Navascues, Carlos Palazuelos, David Perez-Garcia,
  Volkher~B. Scholz, and Reinhard~F. Werner.
\newblock Connes' embedding problem and {T}sirelson's problem.
\newblock {\em Journal of Mathematical Physics}, 52(1):012102, 2011.

\bibitem{KAVRUK2011}
Ali~S. Kavruk, Vern~I. Paulsen, Ivan~G. Todorov, and Mark Tomforde.
\newblock Tensor products of operator systems.
\newblock {\em Journal of Functional Analysis}, 261(2):267 -- 299, 2011.

\bibitem{Kavruk}
Ali~S. Kavruk, Vern~I. Paulsen, Ivan~G. Todorov, and Mark Tomforde.
\newblock Quotients, exactness, and nuclearity in the operator system category.
\newblock {\em Advances in Mathematics}, 235:321--360, 2013.

\bibitem{kavruk2014nuclearity}
Ali~Samil Kavruk.
\newblock Nuclearity related properties in operator systems.
\newblock {\em Journal of Operator Theory}, 71(1):95--156, 2014.

\bibitem{KimPaulsenSchafhauser18}
S.~Kim, V.~I. Paulsen, and C.~Schafhauser.
\newblock A synchronous game for binary constraint systems.
\newblock {\em J. Math. Phys.}, 59(3):032201, 17, 2018.

\bibitem{kirchberg1998c}
Eberhard Kirchberg and Simon Wassermann.
\newblock {C}*-algebras generated by operator systems.
\newblock {\em Journal of Functional Analysis}, 155(2):324--351, 1998.

\bibitem{MuhlySolel1998}
Paul~S. Muhly and Baruch Solel.
\newblock {\em An Algebraic Characterization of Boundary Representations},
  pages 189--196.
\newblock Birkh{\"a}user Basel, Basel, 1998.

\bibitem{ozawa2013connes}
Narutaka Ozawa.
\newblock About the {C}onnes embedding conjecture.
\newblock {\em Japanese Journal of Mathematics}, 8(1):147--183, 2013.

\bibitem{paulsen2002completely}
Vern~I Paulsen.
\newblock {\em Completely bounded maps and operator algebras}, volume~78.
\newblock Cambridge University Press, 2002.

\bibitem{paulsen2016estimating}
Vern~I. Paulsen, Simone Severini, Daniel Stahlke, Ivan~G. Todorov, and Andreas
  Winter.
\newblock Estimating quantum chromatic numbers.
\newblock {\em Journal of Functional Analysis}, 270(6):2188--2222, 2016.

\bibitem{PaulsenTodorovTomforde2011OSS}
Vern~I. Paulsen, Ivan~G. Todorov, and Mark Tomforde.
\newblock Operator system structures on ordered spaces.
\newblock {\em Proceedings of the London Mathematical Society}, 102(1):25--49,
  2011.

\bibitem{PisierKirchbergSimpleProof}
Gilles Pisier.
\newblock A simple proof of a theorem of {K}irchberg and related results on
  {C}*-norms.
\newblock {\em Journal of Operator Theory}, 35(2):317--335, 1996.

\bibitem{Pisier01Similarity}
Gilles Pisier.
\newblock {\em Similarity Problems and Completely Bounded maps}, volume 1618.
\newblock Springer-Verlag, Berlin, 2001.

\bibitem{RUAN1988}
Zhong-Jin Ruan.
\newblock Subspaces of {C$^*$}-algebras.
\newblock {\em Journal of Functional Analysis}, 76(1):217--230, 1988.

\bibitem{RUSSELL2017}
Travis~B. Russell.
\newblock Characterizations of ordered operator spaces.
\newblock {\em Journal of Mathematical Analysis and Applications},
  452(1):91--108, 2017.

\bibitem{RussellGeometry20}
Travis~B. Russell.
\newblock Geometry of the set of synchronous quantum correlations.
\newblock {\em Journal of Mathematical Physics}, 61(5):052201, 2020.

\bibitem{zhangPaulsen16cuntz}
Da~Zheng and Vern paulsen.
\newblock Tensor products of the operator system generated by the {C}untz
  isometries.
\newblock {\em Journal of Operator Theory}, 76(1):67--91, 2016.

\end{thebibliography}

\end{document}